\documentclass[12pt,twoside,a4paper]{article}
\usepackage{amsmath}
\usepackage{amssymb}
\usepackage{amsthm, bbm}
\usepackage{subcaption}
\usepackage{algpseudocode}
\usepackage{mathtools}
\usepackage{algorithmicx}
\usepackage{algorithm}
\usepackage[pdftex]{graphicx}
\usepackage{paralist}% Used for the compactitem environment which makes bullet points with less space between them
\usepackage{parskip}
\usepackage{enumitem}
\usepackage{nicefrac}
\usepackage[numbers, sort]{natbib}
\usepackage[utf8]{inputenc}
\usepackage[T1]{fontenc} 
\usepackage{hyperref}
\usepackage{doi}
\usepackage{bbm}
\usepackage{cleveref}
\usepackage[section]{placeins}
\usepackage[dvipsnames]{xcolor}

\title{Optimality Conditions for Convex Stochastic Optimization Problems in Banach Spaces with Almost Sure State Constraints}
\author{Caroline Geiersbach\thanks{Weierstrass Institute, 10117 Berlin, Germany 
  (\texttt{caroline.geiersbach@wias-berlin.de})}
\and Winnifried Wollner\thanks{Fachbereich Mathematik, Technische Universit\"at Darmstadt, 64293 Darmstadt, Germany, 
  (\texttt{wollner@mathematik.tu-darmstadt.de})}}

\newcommand{\R}{\mathbb{R}}
\newcommand{\cA}{\mathcal{A}}
\newcommand{\cB}{\mathcal{B}}

\renewcommand{\u}{x_1}
\newcommand{\y}{x_2}
\newcommand{\z}{\lambda}
\newcommand{\Y}{X_2}
\newcommand{\Z}{\Lambda}

\newcommand{\U}{X_1}
\newcommand{\by}{\bar{x}_2}
\newcommand{\bx}{\bar{x}}
\newcommand{\bu}{\bar{x}_1}
\newcommand{\blam}{\bar{\lambda}}
\newcommand{\p}{\lambda_e}
\newcommand{\lami}{\lambda_i}
\newcommand{\bp}{\bar{\lambda}_e}
\newcommand{\bl}{\bar{\lambda}_i}
\newcommand{\bz}{\bar{\lambda}}

\newcommand{\E}{\mathbb{E}}
\newcommand{\pP}{\mathbb{P}}
\newcommand{\D}{\text{ d}}

\newcommand{\Yad}{X_{2,\textup{ad}}}
\DeclareMathOperator*{\esssup}{ess\,sup}

\theoremstyle{definition}
\newtheorem{theorem}{Theorem}[section]
\newtheorem{lemma}[theorem]{Lemma}

\newtheorem{definition}[theorem]{Definition}
\newtheorem{corollary}[theorem]{Corollary}
\newtheorem{assumption}[theorem]{Assumption}

\newtheorem{remark}[theorem]{Remark}

\begin{document}
\maketitle
\begin{abstract}
  We analyze a convex stochastic optimization problem where the
  state is assumed to belong to the Bochner space of essentially bounded
  random variables with images in a reflexive and separable Banach space.
  For this problem, we obtain optimality conditions that are, with an
  appropriate model, necessary and sufficient. Additionally, the Lagrange
  multipliers associated with optimality conditions are integrable
  vector-valued functions and not only measures.
  A model problem is given demonstrating the application to
  PDE-constrained optimization under uncertainty with an outlook for further applications.
\end{abstract}

\maketitle
\section{Introduction}
Let $\U$ and $\Y$ be real, reflexive, and separable Banach spaces.
$(\Omega, \mathcal{F}, \pP)$ denotes a complete probability space, where
$\Omega$ represents the sample space, $\mathcal{F} \subset 2^{\Omega}$ is
the $\sigma$-algebra of events on the power set of $\Omega$, and
$\pP\colon \Omega \rightarrow [0,1]$ is a probability measure. We assume
$C_1 \subset \U$ is nonempty, closed, and convex;
$\Yad(\u,\omega) \subset \Y$ is assumed to be nonempty, closed, and
convex for all $\u \in C_1$ and almost all $\omega \in \Omega$.
We are interested
in a convex stochastic optimization problem of the form
\begin{equation}
 \label{eq:model-problem-abstract}
 \begin{aligned}
   \min_{x_1,x_2(\cdot)} \quad &\left\lbrace\E[J(\u, \y(\cdot))]
   = \int_{\Omega} J(\u, \y(\omega)) \D \pP(\omega)\right\rbrace\\
   &\text{s.t.}\quad\left\{\begin{aligned}\u &\in C_1, \\
   \y(\omega) &\in \Yad(\u,\omega) \quad \text{a.s.},
   \end{aligned}\right.
 \end{aligned}
\end{equation}
where $J$ is a convex real-valued mapping.
In this model, the variable $\u$, unlike $\y$, is independent of the
random data. As such, this problem can be interpreted as a
two-stage stochastic optimization problem. In the formulation \eqref{eq:model-problem-abstract}, it is assumed that the function $\omega \mapsto \y(\omega)$ is provided at the outset, which gives all possible decisions for each $\omega$. This viewpoint differs in spirit from a stochastic optimization problem with recourse, where the second-stage ``decision'' $\y$ is made only \textit{after} observing a random element $\omega$. However, under mild assumptions, these problems can be shown to be equivalent to each other; see, e.g., \cite[Section 3]{Rockafellar1976}. This fact is also known as the interchangeability principle for two-stage programming, see \cite[Section 2.3]{Shapiro2009}. 

Such problems are of interest for applications to optimization with
partial differential equations (PDEs) under uncertainty, where the set
to which $\y(\omega)$ belongs includes those states solving a PDE. This
field is a rapidly developing one, with many developments in understanding
the modeling, theory, and design of efficient algorithms; see,
e.g.,~\cite{Conti2008, Kouri2013, Rockafellar2015, Kouri2016, Alexanderian2017, VanBarel2017, Guth2021, Chen2019, Geiersbach2020b}
and the references therein. 
So far, research has mostly been limited to the case where
the control (in our notation, the first-stage variable $\u$) has been
subject to additional constraints. In this case, optimality conditions
have already been established for risk-averse problems
in~\cite{Kouri2018,Kouri2019a}. However, additional constraints on the state
(here, $\y$), beyond a uniquely solvable equation,
have yet to be investigated thoroughly. Although chance
constraints have been handled in such applications,
cf.~\cite{Farshbaf-Shaker2018}, the treatment of pointwise almost sure
constraints on the state appear to be missing from the literature.

As a first step in this treatment, optimality conditions play a central role,
and we pursue this in the current paper.
Pointwise state constraints, without uncertainty, have received some
attention over the last years, see, e.g.,~\cite{Polak:1997} for a theory
of consistent approximations for optimal control problems with ODEs,
or~\cite{Troeltzsch:2010} for the function space analysis for PDE
constrained problems.
For the latter, optimality conditions require Lagrange multipliers coming, in general,
from the non-separable
space of regular Borel measures, see, e.g.,~\cite{Casas1986,CasasBonnans1988}.
Due to the irregular nature of the multipliers, 
penalty~\cite{BergouniouxKunisch:1997,Ulbrich:2011,Hintermueller2014}
and barrier approaches~\cite{SchielaWollner:2008,Schiela:2007}
have been investigated on a function space level.
However, under the mild assumption of bounded,
rather than square integrable, problem data, it could be shown that
multipliers of a model problem can be found in a more regular,
separable, space, see~\cite{CasasMateosVexler2014,BrennerSung2017,BrennerSungWollner:2020b}. Similar observations are true for parabolic optimization problems, see~\cite{ChristofVexler:2021}.

In this paper, we are focused on obtaining optimality conditions in the
case where $\y$ belongs to the Bochner space $L^\infty(\Omega,\Y)$. This
choice is motivated by the goal of including problems where there is an
almost sure bound such as
\begin{equation*}
  \y(\omega) \leq_K \psi(\omega),
\end{equation*}
where $\psi \in L^\infty(\Omega,\Y)$ and $\leq_K$ represents a partial
order on $X_2$. An example with this type of inequality is given in
\cref{subsection:Example}. The choice of $L^p(\Omega,\Y)$ for
$p<\infty$ is not appropriate, as the cone
$\{v \in L^p(\Omega,\Y): v(\omega) \leq_K 0\}$ contains no interior points;
this property is especially important in the establishment of Lagrange
multipliers for our application. Therefore, we will view the problem
presented in~\eqref{eq:model-problem-abstract} in the framework of
two-stage stochastic optimization (for an introduction,
see~\cite{Shapiro2009, Pflug2014}). This framework allows us to
generalize results from a series of papers by Rockafellar and
Wets~\cite{Rockafellar1975, Rockafellar1976,Rockafellar1976a,Rockafellar1976b},
who established optimality theory of general convex stochastic
optimization problems with states belonging to the space
$L^\infty(\Omega,\R^n)$. As the class of problems we are treating involve
equality constraints, we include that theory here, which is not covered
by the papers~\cite{Rockafellar1975, Rockafellar1976,Rockafellar1976a,Rockafellar1976b}.
Additionally, we emphasize that care must be taken in our setting,
where the random variables are vector-valued. 

While much of the literature on which we base our analysis is
classical, we note that the study of problems of the form
\eqref{eq:model-problem-abstract} remain an active area of research
thanks to the difficulties presented in specific applications. These
difficulties are present not only in optimal control problems with
PDEs but also those in mathematical finance, see for instance
\cite{Pennanen2018, Pennanen2018a}. In \cite{Pennanen2018}, the
authors develop duality theory in the same spirit as we do, focusing
on the case of where an integral functional is defined over variables
$x:\Omega \times \R \rightarrow \R^d$ of bounded variation
with finite dimensional image. Also of
relevance are the recent works \cite{Leclere2014, Leclere2019}. The
first of these works also considers optimality conditions for problems
similar to ours, although the exposition is limited to random vectors,
i.e., with finite-dimensional images. The latter work also includes vector-valued random variables and focuses on a relaxation of problems like \eqref{eq:model-problem-abstract}, where the almost sure constraint is replaced by its conditional expectation. This is done in view of justifying tractable decomposition methods with subproblems that are easier to solve.

We will proceed by introducing our notation and proving essential results
about subdifferentiability of convex integral functionals on the space 
$L^\infty(\Omega,X)$ in \cref{sec:background}. The core of the paper
is contained in \cref{sec:problem-formulation}, where we use the
perturbation approach to show the existence of saddle points for a
suitably tailored generalized Lagrangian.
This approach allows us to look for Lagrange
multipliers in the space $L^1(\Omega,X^*)$, instead of $(L^\infty(\Omega,X))^*$,
and provide Karush--Kuhn--Tucker conditions for our problem.
In \cref{sec:model-problem}, we show an application to
PDE-constrained optimization under uncertainty. Here we will see that
while a direct addition of randomness to a typical model problem does not
fit into our theory, a suitable penalization does. This allows the
approximation of PDE-constrained problems with almost sure
state constraints by a sequence of problems admitting multipliers in
$L^1(\Omega,X^*)$.
We close with some remarks
in \cref{sec:Conclusion}.

\section{Background and Notation}
\label{sec:background}
Throughout, we shall employ the following notation. We assume that $X$ is
a real, reflexive, and separable space; the dual is denoted by $X^*$ and the
canonical dual pairing is written as $\langle \cdot, \cdot \rangle_{X^*,X}$.
Given a set $C \subset X$, $\delta_C$ denotes the indicator function,
where $\delta_C(x) = 0$ if $x \in C$ and $\delta_C(x) =\infty$ otherwise.
The interior of a set $C$ is denoted by $\textup{int}\,C.$
The sum of two sets $A$ and $B$ with $\lambda \in \R$ is given by
$A+\lambda B:=\{ a+\lambda b: a \in A, b \in B\}.$
We recall that for a proper function $h:X \rightarrow \R \cup \{ \infty \}$,
the subdifferential (in the sense of convex analysis) is the set-valued
operator defined by
\begin{equation*}
  \partial h: X \rightrightarrows X^*: x \mapsto \{ q \in X^*: \langle q, y - x \rangle_{X^*,X} + h(x) \leq h(y) \quad  \forall y \in X \}.
\end{equation*} 
The domain of $h$ is denoted by
$\textup{dom}(h):= \{ x \in X: h(x) < \infty \}.$
Given $K\subset X$, the support function of $K$ is denoted by $\sigma(K,v) := \sup_{x \in K} \langle v, x \rangle_{X^*,X}$ for all $v \in X^*$. A strongly $\pP$-measurable mapping from $\Omega$ to a Banach space $X$ is
referred to as an $X$-valued random variable. As the underlying
probability space is considered fixed, we will frequently write simply
``measurable'' instead of ``$\pP$-measurable.'' Additionally, since we
only consider separable spaces, weak and strong measurability coincide,
in which case we can simply refer to measurability of a random
variable.\footnote{More precisely, for $y:\Omega \rightarrow X$, the
  following assertions are equivalent: 1) $y$ is strongly measurable
  and 2) $y$ is separably-valued and
  measurable~\cite[Corollary 1.1.10]{Hytoenen2016}.}

Given a Banach space $X$ equipped with the norm $\lVert \cdot \rVert_X$,
the Bochner space $L^r(\Omega, X)$ is the set of all (equivalence classes
of) $X$-valued random variables having finite norm, where the norm is
given by
\begin{equation*}
  \lVert y \rVert_{L^r(\Omega,X)}:= \begin{cases}
  (\int_\Omega \lVert y(\omega) \rVert_X^r \D \pP(\omega))^{1/r}, \quad &1 \leq r < \infty,\\
  \esssup_{\omega \in \Omega} \lVert y(\omega) \rVert_X, \quad &r=\infty.
\end{cases}
\end{equation*} 
An $X$-valued random variable $x$ is Bochner integrable if there exists a
sequence $\{ x_n\}$ of $\pP$-simple functions $x_n:\Omega \rightarrow X$
such that $\lim_{n \rightarrow \infty} \int_{\Omega} \lVert x_n(\omega)-x(\omega) \rVert_X \D \pP(\omega) = 0$.
The limit of the integrals of $x_n$ gives the Bochner integral
(the expectation), i.e., 
\begin{equation*}
  \E[x]:=\int_\Omega x(\omega) \D \pP(\omega) = \lim_{n \rightarrow \infty} \int_{\Omega} x_n(\omega) \D \pP(\omega).
\end{equation*}
Clearly, this expectation is an element of $X$.

Recall that a property is said to hold \textit{almost surely} (a.s.)
provided that the set (in $\Omega$) where the 
property does not hold is a set of measure zero. As an example, two random
variables $\xi, \xi'$ are said to be equal almost surely, $\xi = \xi'$
a.s., if and only if
$\pP(\{ \omega \in \Omega: \xi(\omega) \neq \xi'(\omega)\}) = 0$,
or equivalently,
$\pP(\{ \omega \in \Omega: \xi(\omega) = \xi'(\omega)\}) = 1$. 

\subsection{Subdifferentiability of convex integral functionals on $L^\infty(\Omega,X)$}
\label{subsec:subdifferentiability-Linf}
In order to obtain optimality conditions for a problem of the
form~\eqref{eq:model-problem-abstract}, we will first provide some
background on convex integral functionals defined on the space
$L^\infty(\Omega,X)$, where $X$ is assumed to be a real, reflexive, and
separable Banach space.\footnote{While we continue using the probability
  space $(\Omega,\mathcal{F}, \pP)$, the results of this section also
  hold for more general $\sigma$-finite complete measure spaces.} %from~\cite{Rockafellar1968, Rockafellar1971, Rockafellar1971a}. 
We denote the $\sigma$-algebra of Borel sets on $X$ by $\mathcal{B}$.
We study convex functionals of the form
\begin{equation}
  \label{eq:integral-functional}
  I_f(x) := \int_{\Omega} f(x(\omega),\omega) \D\pP(\omega),
\end{equation}
where $x:\Omega \rightarrow X$ and
$f:X \times \Omega  \rightarrow \R \cup \{ \infty \}$.
The function $f$ is called a convex integrand if
$f_\omega: = f(\cdot,\omega)$ is convex for every $\omega$ (it is no
loss of generality to redefine a functional that is only convex for
almost every $\omega$).
This integrand is called \textit{normal} if it is not identically
infinity, it is $(\mathcal{B}\times\mathcal{F})$-measurable, and
$f_\omega$ is lower semicontinuous in $X$ for each $\omega \in \Omega$.
An example of a function that is normal is one that is finite everywhere
and Carath\'eodory, meaning $f$ measurable in $\omega$ for fixed $x$
and continuous in $x$ for fixed $\omega$. %\cite{Rockafellar1971}, pp. 221-222 
Normality of $f$ makes it superpositionally measurable, meaning
$\omega \mapsto f(x(\omega),\omega)$ is measurable if
$x:\Omega \rightarrow X$ is measurable; see, e.g.,~\cite[Lemma 8.2.3]{Aubin1990}. 

If $\omega \mapsto f(x(\omega),\omega)$ is majorized by an integrable
function $g$, i.e., $|f(x(\omega),\omega)| \leq g(\omega)$ a.s., then the integral functional~\eqref{eq:integral-functional} is
finite; if no such majorant exists, by convention, we set $I_f(x)=\infty$.
The conjugate of the normal convex integrand $f_\omega$ is the function
$f_\omega^*$ defined on $X^*$ by %Rockafellar1968, Lemma 5, 
\begin{equation*}
\label{eq:conjugate-integrand-definition}
f_\omega^*(x^*) := \sup_{x \in X} \{ \langle x^*,x \rangle_{X^*,X} - f_\omega(x)\}.
\end{equation*}
By~\cite[Proposition 6.1]{Levin1975}, $f_\omega^*$ is a normal convex
integrand and $(f_\omega^*)^*=f_\omega.$
We recall, see, e.g.,~\cite[Proposition 6.5.4]{Aubin1990} 
that if $f_\omega$ is convex, 
\begin{equation}
\label{eq:conjugate-subdifferential}
x^* \in \partial f_\omega(x) \quad \text{if and only if} \quad \langle x^*,x \rangle_{X^*,X} = f_\omega(x) + f_\omega^*(x^*).
\end{equation}
%$f^*$ is a normal convex integrand provided there is a $x^* \in X^*$ such that $f^*(x^*,\omega)$ for every $\omega$. In the case where $f$ is convex, $f$ is the integrand conjugate of $g$~\cite[Proposition 2]{Rockafellar1971}.

Even if the Radon--Nikodym property is satisfied for $X$, there is not
generally an isometry between $(L^\infty(\Omega,X))^*$ and $L^1(\Omega,X^*)$.
However, there is a useful decomposition on this dual space; namely,
elements can be decomposed into absolutely continuous and singular parts.
A continuous linear functional $v \in (L^\infty(\Omega,X))^*$ of the form
\begin{equation*}
  v(x) = \int_{\Omega} \langle x^*(\omega), x(\omega) \rangle_{X^*,X} \D \pP(\omega)
\end{equation*}
for some $x^* \in L^1(\Omega, X^*)$ is said to be absolutely continuous.
These functionals form a closed subspace of $(L^\infty(\Omega,X))^*$ that
is isometric to $L^1(\Omega,X^*)$. This subspace has a complement consisting
of singular functionals, defined next.
\begin{definition}
\label{def:singular-functionals}
A functional $v^\circ \in (L^\infty(\Omega,X))^*$ is called
\textit{singular (relative to $\pP$)} if there exists a sequence
$\{ F_n\} \subset  \mathcal{F}$ with $F_{n+1} \subset F_n$ for all $n$,
$\pP(F_n) \rightarrow 0$ as $n\rightarrow \infty$, and $ v^\circ(x) = 0$
for all $x \in L^\infty(\Omega, X)$ satisfying $x(\omega) \equiv 0$ for
almost all $\omega \in F_n$ for some $n$. 
\end{definition}
The following decomposition result was proven
in~\cite[Appendix 1, Theorem 3]{Ioffe1972} (with a slight correction to the original proof in \cite{Levin1974}).
\begin{theorem}[Ioffe and Levin]
Each functional $v^* \in (L^\infty(\Omega,X))^*$ has a unique decomposition 
\begin{equation}
\label{eq:decomposition-dual-space}
v^* = v + v^\circ,
\end{equation}
where $v$ is absolutely continuous, $v^\circ$ is singular relative to
$\pP$, and 
\begin{equation*}
  \lVert v^* \rVert_{(L^\infty(\Omega,X))^*} = \lVert v \rVert_{(L^\infty(\Omega,X))^*}+\lVert v^\circ \rVert_{(L^\infty(\Omega,X))^*}.
\end{equation*}
\end{theorem}
The next result characterizes the convex conjugate of a
functional~$I_f$ defined
on $L^\infty(\Omega,X)$. By definition, the convex functional
on $(L^\infty(\Omega,X))^*$ that is conjugate to $I_f$ is given by
\begin{equation}
\label{eq:If-star}
I_f^*(v^*) := \sup_{z \in L^\infty(\Omega,X)} \{ v^*(z) - I_{f}(z)\}.
\end{equation}
This functional is closely related to the integral functional $I_{f^*}$,
where $f^*$ denotes the conjugate of the normal convex integrand $f$ as
before.
The following theorem relates $I_f^*$ to $I_{f^*}$ and was proven for
$X = \R^n$ in~\cite[Theorem 1]{Rockafellar1971a} and later for separable
(generally non-reflexive) Banach spaces in~\cite[Theorem 6.4]{Levin1975}.
 
\begin{theorem}[Levin]
\label{thm:conjugate-Linfty-functional}
Assume $f$ is a normal convex integrand and $I_f(x) < \infty$ for some $x \in L^\infty(\Omega,X)$.
%, and $I_{f^*}(x^*) < \infty$ for at least one $x^* \in L^1(\Omega,X^*)$. 
Then the functional $I_f^*$ can be represented by the decomposition
\begin{equation}
\label{eq:representation-dual-integral}
I_f^*(v^*) = I_{f^*}(x^*) + \sigma({\textup{dom} (I_f)},v^\circ),
\end{equation}
where $x^* \in L^1(\Omega, X^*)$ corresponds to the absolutely
continuous part of $v^*$ and $v^\circ  \in (L^\infty(\Omega,X))^*$ corresponds
to the singular part of $v^*$, and $\sigma({\textup{dom} (I_f)},v^\circ)$
denotes the support functional of $\textup{dom}(I_f)$ in $v^\circ.$ 
\end{theorem}

\begin{remark}
\label{rem:integral-well-defined}
The assumption that $I_f(x)<\infty$ for some $x \in L^\infty(\Omega,X)$
implies that $I_{f^*}$ is a well-defined convex functional on
$L^1(\Omega,X^*)$ with values in $\R \cup \{ \infty \}$. Indeed, since $f_\omega^*$ and
  $f_\omega$ are conjugate to each other, we have for all $\omega$ and
  all $x^* \in L^1(\Omega,X^*)$ %that
\begin{equation}
\label{eq:fomega-star-ineq}
f_{\omega}^*(x^*(\omega)) \geq \langle x^*(\omega),x(\omega)\rangle_{X^*,X} - f_\omega(x(\omega)).
\end{equation}
The right side is integrable by assumption, so $I_{f^*} > -\infty$ on
$L^1(\Omega,X^*)$. If one additionally has $I_{f^*}(x^*) < \infty$ for
some $x^* \in L^1(\Omega, X^*)$, then one shows in the same way that $I_f$
is well-defined on $L^\infty(\Omega,X)$ with values in $\R \cup \{ \infty \}.$
\end{remark}

The following result gives a bound on the singular element $v^\circ$. 

\begin{theorem}
\label{thm:bounds-support-function}
Let $f$ be a normal convex integrand. Let $\bar{x} \in L^\infty(\Omega,X)$
be such that there exists $r >0$ and an integrable function $k_r$ of
$\omega$ satisfying $f_\omega(x(\omega)) \leq k_r(\omega)$ as long as $\lVert x- \bx\rVert_{L^\infty(\Omega,X)} < r$.
%By convexity, let $x$ be such that $\|x\|_{L^\infty(\Omega,X)}< r.$ Then $2 f_\omega(\bx(\omega)) \leq f_\omega(\bx(\omega)+x(\omega)) + f_\omega(\bx(\omega)-x(\omega))$ and therefore $f_\omega(\bx(\omega)+x(\omega)) \geq 2f_\omega(\bx(\omega)) - k_r(\omega).$ 
%\footnote{Related conditions to this one can be found in~\cite[Theorem 5.1]{Levin1975}.}
Then the conjugate integrand $f_\omega^*(x^*(\omega))$ is majorized by an
integrable function of $\omega$ for at least one $x^* \in L^1(\Omega, X^*)$.
Additionally, $I_f$ is continuous at $x$ as long as
$\lVert x-\bar{x} \rVert_{L^{\infty}(\Omega,X)} < r$; in this case,  
the function $\sigma({\textup{dom} (I_f)},\cdot)$ given
in~\eqref{eq:representation-dual-integral} can be bounded as follows:
\begin{equation}
\label{eq:bound-support-function}
\sigma(\textup{dom} (I_f), v^\circ) \geq v^\circ(\bar{x}) + r \rVert v^\circ \rVert_{(L^\infty(\Omega,X))^*}.
\end{equation}
\end{theorem}

\begin{proof}
We proceed as in~\cite[Theorem 2]{Rockafellar1971a}, making modifications
for the infinite-dimensional setting.
Using~\eqref{eq:conjugate-subdifferential}, we have
\begin{equation*}
  \partial f_{\omega}(\bx(\omega))= \{ q \in X^* \colon \langle q,\bx(\omega)\rangle_{X^*,X} = f_{\omega}(\bx(\omega)) + f^*_\omega(q) \}.
\end{equation*}
We show that the set-valued map
$\omega \mapsto \partial f_{\omega}(\bx(\omega))$ is measurable by first
proving that the support function of $\partial f_{\omega}(\bx(\omega))$ is
measurable.  Since $f_\omega$ is convex and finite on a neighborhood of
$\bx(\omega)$, it is continuous at $\bx(\omega)$, so the set
$\partial f_{\omega}(\bx(\omega))$ is a nonempty, convex, and weakly*
compact subset of $X^*$ and $f_\omega$ is Hadamard directionally
differentiable in $\bx(\omega)$~\cite[Proposition 2.126]{Bonnans2013}.
Since $X$ is reflexive, the support function of
$\partial f_{\omega}(\bx(\omega))$ in $x$ is given by %\cite[p.62]{Aubin1990}
\begin{equation*}
  \sigma(\partial f_{\omega}(\bx(\omega)),x) = \sup_{q \in \partial f_{\omega}(\bx(\omega))} \langle x,q \rangle_{X,X^*}.
\end{equation*}
Thus, since $f_\omega$ is convex, we have
\begin{equation}
\label{eq:bound-support-function-by-function-values}
\begin{aligned}
\sigma(\partial f_{\omega}(\bx(\omega)),x) &= f'_\omega(\bx(\omega); x) %= \lim_{t \rightarrow 0^+} \frac{1}{t} \Big( f_{\omega}(\bx(\omega) + tx) - f_\omega(\bx(\omega)) \Big) \\
 = \inf_{t \geq 0}  \frac{1}{t} \Big( f_{\omega}(\bx(\omega) + tx) - f_\omega(\bx(\omega)) \Big)\\
&\leq f_{\omega}(\bx(\omega) + x) - f_\omega(\bx(\omega)).
\end{aligned}
\end{equation}
Measurability of $\omega \mapsto \sigma(\partial f_{\omega}(\bx(\omega)),x)$
follows from the fact that the limit of a sequence of measurable functions
is measurable~\cite[p.~307]{Aubin1990}. Since $X$ is reflexive and
separable, we obtain from~\cite[Theorem 8.2.14]{Aubin1990} that
$\omega \mapsto \partial f_{\omega}(\bx(\omega))$ is measurable. The
measurable selection theorem~\cite[Theorem 8.1.3]{Aubin1990} guarantees
the existence of a measurable function $x^*:\Omega \rightarrow X^*$ such
that $x^*(\omega) \in \partial f_{\omega}(\bx(\omega))$ for every
$\omega \in \Omega.$
From~\eqref{eq:bound-support-function-by-function-values} it follows for
this $x^*$ that
\begin{equation*}
  \langle x, x^*(\omega) \rangle_{X,X^*} \leq \sigma(\partial f_{\omega}(\bx(\omega)),x)\leq f_{\omega}(\bx(\omega) + x) - f_\omega(\bx(\omega)).
\end{equation*}
As long as $x \in X$ satisfies $\lVert x \rVert_{X} < r$, we obtain by
assumption that
\begin{equation}
\label{eq:majorant-for-measurable-selection}
r \lVert x^*(\omega)\rVert_{X^*} =\sup_{x:\lVert x \rVert_{X}\leq r} \langle x, x^*(\omega) \rangle_{X,X^*}  \leq k_r(\omega) - f_\omega(\bx(\omega)).
\end{equation}
The right-hand side of~\eqref{eq:majorant-for-measurable-selection} is
integrable, thus $x^*\in L^1(\Omega, X^*).$

Now, by \eqref{eq:conjugate-subdifferential}, we have for this $x^* \in L^1(\Omega,X^*)$
$$f_\omega^*(x^*(\omega)) = \langle x^*(\omega), \bar{x}(\omega) \rangle_{X^*,X} - f_\omega(\bar{x}(\omega)),$$
from which we immediately obtain that $f^*_\omega(x^*(\omega))$ is majorizable.

For any $x\in L^\infty(\Omega,X)$ with
$\lVert x-\bx \rVert_{L^\infty(\Omega, X)} < r$, we get $$I_f(x) \leq
\int_{\Omega} k_r(\omega) \D \pP(\omega) < \infty,$$ implying $I_f(x)$
is bounded above and continuous at $x$, i.e., $x \in \textup{dom} (I_f)$, so 
\begin{align*}
  \sigma(\textup{dom} (I_f), v^\circ)
  &= \sup_{x \in \textup{dom} (I_f)} v^\circ(x) \geq \sup_{x: \lVert x
    - \bx \rVert_{L^\infty(\Omega, X)} < r}\!\! v^\circ(x) =  v^\circ(\bx) + r \lVert v^\circ \rVert_{(L^\infty(\Omega,X))^*}.
\end{align*}
This is the expression~\eqref{eq:bound-support-function}, so the proof is
complete.
\end{proof}

The next two results can be obtained as
in~\cite[Corollary 2A, 2C]{Rockafellar1971a}.

\begin{corollary}
\label{cor:duality-conjugate-functions}
Assume $f$ is a normal convex integrand and $f(x(\omega),\omega)$ is an
integrable
function of $\omega$ for every $x \in L^\infty(\Omega,X)$.
Then $I_f$ and $I_{f^*}$ are
well-defined convex functionals on $L^\infty(\Omega,X)$ and $L^1(\Omega, X^*)$,
respectively, that are conjugate to each other in the sense that
\begin{align*}
I_{f^*}(x^*) &= \sup_{x\in L^\infty(\Omega,X)} \left\lbrace \int_{\Omega} \langle x^*(\omega),x(\omega)\rangle_{X^*,X} \D \pP(\omega) - I_f(x)\right\rbrace,\\
I_{f}(x) &= \sup_{x^* \in L^1(\Omega,X^*)} \left\lbrace \int_{\Omega} \langle x^*(\omega),x(\omega)\rangle_{X^*,X}  \D \pP(\omega) - I_{f^*}(x^*)\right\rbrace.
\end{align*}
Furthermore, if $v^*$ is an absolutely continuous functional corresponding
to a function $x^* \in L^1(\Omega, X^*)$, then $I_f^*(v^*) = I_{f^*}(x^*)$,
while $I_f^*(v^*) = \infty$ for any $v^*$ that is not absolutely continuous.
\end{corollary}

\begin{proof}
Since $f(x(\omega),\omega)$ is integrable for all $x$, it is also integrable for $x \equiv 0$.
Now, by \cite[Theorem 5.1]{Levin1975}, this implies the existence of a $r>0$ and integrable function $k_r$ such that $f_\omega(0+ x)\leq k_r(\omega) $ a.s.~for all $x \in X$ such that $\lVert x \rVert_{X} \leq r.$
\cref{thm:bounds-support-function} gives the bound \eqref{eq:bound-support-function}, which in combination with \eqref{eq:representation-dual-integral} gives the conclusion with $r=\infty$.
\end{proof}

\begin{corollary}
\label{cor:characterization-subdifferentials}
Let $f$ and $\bx$ satisfy the assumptions of
\cref{thm:bounds-support-function}. Then $v^* \in (L^\infty(\Omega, X))^*$
is an element of $\partial I_f(\bx)$ if and only if
\begin{equation}
\label{eq:subdifferential-a.s.}
x^*(\omega) \in \partial f_\omega(\bx(\omega)) \quad \text{a.s.},
\end{equation}
where $x^* \in L^1(\Omega,X^*)$ corresponds to the absolutely
continuous part $v$ of $v^*$ and the singular part $v^\circ$ of $v^*$ satisfies
$\sigma(\textup{dom} (I_f), v^\circ) = v^\circ(\bx).$ Moreover,
$\partial I_f(\bx)$ can be identified with a nonempty, weakly compact
subset of $L^1(\Omega, X^*)$. In particular, $v^*$ belongs to
$\partial I_f(\bx)$ if and only if $v^\circ \equiv 0$ and $v=x^*$
satisfies~\eqref{eq:subdifferential-a.s.}.
\end{corollary}

\begin{proof}
By \cref{thm:bounds-support-function}, $I_f$ is finite on a
neighborhood of $\bx$ and is continuous at $\bx$; it is naturally convex
by convexity of $f$. In particular $\partial I_f(\bar{x})$ is a nonempty,
weakly* compact subset of $(L^\infty(\Omega,X))^*$.%; cf.~\cite[Proposition 2.126]{Bonnans2013}.

Using~\eqref{eq:decomposition-dual-space}, notice that
by~\eqref{eq:conjugate-subdifferential} $v^* \in \partial I_f(\bx)$
if and only if 
\begin{equation*}
\begin{aligned}
0 &=I_f^*(v^*) +  I_f(\bx) - v^*(\bx)\\
&= \sup_{z \in L^\infty(\Omega,X)} \{ v^\circ(z) +v(z) - I_f(z) \} +  I_f(\bx) - v^\circ(\bx) - v(\bx),
\end{aligned}
\end{equation*}
i.e., the supremum is attained in $z = \bx$.
Now, by \cref{thm:conjugate-Linfty-functional}
and~\eqref{eq:fomega-star-ineq} one has
\begin{equation*}
  \begin{aligned}
    v^\circ(\bx) +v(\bx) - I_f(\bx) &= I_f^*(v^*)= I_{f^*}(x^*) + \sigma({\textup{dom} (I_f)},v^\circ)\\
    &\ge v(\bx) - I_f(\bx) + \sigma({\textup{dom} (I_f)},v^\circ)
  \end{aligned}
\end{equation*}
and thus $v^\circ(\bx) \ge \sigma(\textup{dom}(I_f),v^\circ)$.
By~\eqref{eq:bound-support-function}, this can be the case
if and only if $v^\circ \equiv 0$.
Thus using~\eqref{eq:representation-dual-integral}, we have that
\begin{align}
0 &=I_f^*(v^*) +  I_f(\bx) - \sigma(\textup{dom} (I_f),v^\circ) -  \int_\Omega \langle x^*(\omega), \bx(\omega) \rangle_{X^*,X} \D \pP(\omega) \nonumber \\
&=I_{f^*}(x^*) + I_f(\bx)- \int_\Omega \langle x^*(\omega), \bx(\omega) \rangle_{X^*,X} \D \pP(\omega).\nonumber\\
 &= \int_{\Omega} f^*_{\omega}(x^*(\omega)) + f_\omega(\bx({\omega})) -  \langle x^*(\omega), \bx(\omega) \rangle_{X^*,X} \D \pP(\omega).\label{eq:integral-equal-zero}
\end{align}

Notice that the integrand in~\eqref{eq:integral-equal-zero} is non-negative
by definition of the conjugate $f_\omega^*$, i.e.,~\eqref{eq:fomega-star-ineq}.
We obtain that the
integrand~\eqref{eq:integral-equal-zero} is almost surely equal to zero
and, recalling the equivalent expression for the
subdifferential~\eqref{eq:conjugate-subdifferential},~\eqref{eq:subdifferential-a.s.} follows.  

For the second claim, since $X$ is reflexive and separable, we have the
isometric isomorphism~\cite[Corollary 1.3.22]{Hytoenen2016}
\begin{equation}
\label{eq:isometric-isomorphism}
(L^1(\Omega,X^*))^* \simeq L^\infty(\Omega,X^{**}) = L^\infty(\Omega,X).
\end{equation}
Since all elements of the subdifferential in fact belong to
$L^1(\Omega, X^*)$, $\partial I_f(\bar{x})$ can be identified with a
subset of $L^1(\Omega,X^*)$. The fact that this subset is weakly compact in $L^1(\Omega,X^*)$ 
follows from~\eqref{eq:isometric-isomorphism} and the fact that
$\partial I_f(\bar{x})$ is weakly* compact in $(L^\infty(\Omega,X))*$.
\end{proof}

\section{Lagrangian Duality and Optimality Conditions}
\label{sec:problem-formulation}
In everything that follows, we will consider the case where the admissible
set of states from~\eqref{eq:model-problem-abstract} contains both an
equality and inequality (cone) constraint. Let $W$ and $R$ be real,
reflexive, and separable Banach spaces. The equality and inequality
constraint are defined by the mappings
$e:\U\times \Y\times \Omega \rightarrow W$ and $i:\U \times \Y \times \Omega \rightarrow R$,
respectively. Given a cone $K \subset R$, the partial order $\leq_K$ is
defined by $r \leq_K 0  :\Leftrightarrow -r \in K$,
or equivalently, $r \geq_K 0$ if and only if $r \in K$. 
The corresponding dual cone is denoted by
$K^\oplus := \{  r^* \in R^*: \langle r^*,r \rangle_{R^*,R} \geq 0 \,\forall r \in K \}.$
The admissible set takes the form
\begin{equation*}
  \Yad(\u,\omega):=\{ \y \in C_2: e(\u,\y,\omega) = 0, i(\u,\y,\omega) \leq_K 0\}.
\end{equation*}
Additionally, we assume that the integrand takes the form
\begin{equation}
\label{eq:multistage-integrand}
J(\u,\y):=J_1(\u) + J_2(\u,\y).
\end{equation}
The problem introduced in~\eqref{eq:model-problem-abstract} is now defined
over $x:=(x_1,x_2) \in X:=\U \times L^\infty(\Omega,\Y)$ by
\begin{equation}
 \label{eq:model-problem-abstract-long}\tag{$\textup{P}$}
 \begin{aligned}
 \min_{x \in X} \quad &\{j(x):=J_1(\u) +\E[J_2(\u, \y(\cdot))]\}\\
 &\text{s.t.} \quad\left\{\begin{aligned}
     x_1 &\in C_1,\\
     \y(\omega) &\in C_2 \text{ a.s.},\\
     e(\u,\y(\omega),\omega) &= 0 \text{ a.s.},\\
     i(\u, \y(\omega), \omega) &\leq_K 0 \text{ a.s.}
   \end{aligned}\right.
 \end{aligned}
\end{equation}
We make the following assumptions about Problem~\eqref{eq:model-problem-abstract-long}.

\begin{assumption}
\label{assumption:general-problem}
Let $C_1 \subset \U$ and $C_2 \subset \Y$ be nonempty, closed, and convex
sets and let $K \subset R$ be a nonempty, closed, and convex cone. 
Assume that the integrand $(\u,\y) \mapsto J(\u,\y)$ is convex on $\U\times \Y$
and is everywhere defined and finite.
Moreover, assume that for every $r > 0$, there exist 
$a_r >0$ such that for any $\lVert x_1\rVert_{X_1} + \lVert x_2\rVert_{X_2} \leq r$, it holds that
\[
|J_2(\u,\y)| \le a_r.
\]
%for all $\|\y\|_{\Y} \le r$.
Assume $e(\u,\y,\omega)$ is continuous and linear in $(\u,\y)$ and $i(\u,\y,\omega)$ is continuous and $K$-convex\footnote{$K$-convexity of $i(\cdot,\cdot,\omega)$ means 
\begin{equation*}
i(\lambda x_1 + (1-\lambda)\hat{x}_1, \lambda x_2 + (1-\lambda)\hat{x}_2,\omega) \leq_K \lambda i(x_1,x_2,\omega) + (1-\lambda)i(\hat{x}_1, \hat{x}_2,\omega)
\end{equation*}
for all $(x_1,x_2),(\hat{x}_1,\hat{x}_2) \in X_1 \times X_2$,
  and $\lambda \in (0,1)$.
} in $(\u,\y)$;  $e(\u,\y,\omega)$ and
$i(\u,\y,\omega)$ are measurable and for every $r>0$ there exist 
$b_{r,e} > 0$ and $b_{r,i}>0$ such that for any $\lVert x_1\rVert_{X_1} + \lVert x_2\rVert_{X_2} \leq r$, it holds 
\begin{equation*}
\label{eq:growth-condition-constraints}
  \|e(x_1,x_2,\omega)\|_{W} \le b_{r,e}, \quad \|i(x_1,x_2,\omega)\|_{R} \le b_{r,i} \quad \text{a.s.}
\end{equation*}

\end{assumption}

\begin{remark}
By  \cref{assumption:general-problem}, the mappings $J_2$, $e$,
and $i$ are Carath\'eodory and thus for measurable $x_1: \Omega \rightarrow X_1$ 
and $x_2:\Omega \rightarrow X_2$, the mappings
\[
\omega \mapsto J_2(\u(\omega),\y(\omega)),\quad \omega \mapsto e(\u(\omega),\y(\omega),\omega),\quad \omega \mapsto i(\u(\omega),\y(\omega),\omega) 
\]
are measurable, see~\cite[Corollary~8.2.3]{Aubin1990}.
The respective growth conditions assert that if additionally $x_1: \Omega \rightarrow X_1$ and $x_2:\Omega \rightarrow X_2$ are essentially bounded, we have
\begin{gather*}
J_2(\u(\cdot),\y(\cdot)) \in L^\infty(\Omega), \\ e(\u(\cdot),\y(\cdot),\cdot) \in L^\infty(\Omega,W), \quad i(\u(\cdot),\y(\cdot),\cdot) \in L^\infty(\Omega,R).
\end{gather*}
%The superposition with respect to the first argument is needed in \cref{thm:KKT-basic-conditions} when applying the subdifferential.
For more on growth conditions, see, e.g.,~\cite[Section 3.7]{AppellZabrejko:1990}.
%If $\y \in L^\infty(\Omega,\Y)$, then $\omega \mapsto J_2(\u,\y(\omega))$
%is measurable and even integrable by
% \cref{assumption:general-problem}; additionally, it follows that
%$e(\u, \y(\omega), \omega)$ and $i(\u, \y(\omega), \omega)$ are measurable
%and integrable. If $J$ were to also depend on $\omega$, i.e.,
%$J(\u,\y,\omega)$, it needs to be additionally assumed that the mapping
%$\omega \mapsto J(\u,\y,\omega)$ is integrable. For our applications, this
%extra variable is not needed and so we dispense with this detail here. 
\end{remark}

To obtain optimality conditions, it is natural to define the Lagrangian
\begin{align*}
\mathbb{L}(x,\tilde{\lambda}) = j(x) &+\langle \tilde{\lambda}_e, e(\u,\y(\cdot),\cdot)\rangle_{(L^\infty(\Omega,W))^*,L^\infty(\Omega,W)} \\
& + \langle \tilde{\lambda}_i, i(\u, \y(\cdot),\cdot)\rangle_{(L^\infty(\Omega,R))^*,L^\infty(\Omega,R)}.
\end{align*}
However, $\tilde{\lambda}_e$ and $\tilde{\lambda}_i$ do not have natural
representations in their corresponding dual spaces. We will show that under
certain conditions, Lagrange multipliers can be found in the space
$L^1(\Omega,W^*)$ for the equality constraint and $L^1(\Omega,R^*)$ for the
inequality constraint. To this end, we will show when saddle points of a
(generalized) Lagrangian exist in
\cref{sec:existence-saddle-points}.
This will allow us to formulate Karush--Kuhn--Tucker (KKT) conditions for
Problem~\eqref{eq:model-problem-abstract-long} in
\cref{subsec:KKT}.

\subsection{The Generalized Lagrangian and Existence of Saddle Points}
\label{sec:existence-saddle-points}
In this section, we define a generalized Lagrangian and discuss the
existence of saddle points for Problem~\eqref{eq:model-problem-abstract-long}.
We will use the perturbation approach, meaning that we first introduce the
perturbed problem
\begin{equation}
\label{eq:model-PDE-UQ-state-constraints-perturbed}
\tag{$\textup{P$^{u}$}$}
 \begin{aligned}
  \min_{x \in X} \quad &\varphi(x,u) \\
  &\text{s.t. } \left\{\begin{aligned}
  \u &\in C_1,\\
  \y(\omega) &\in C_2 \text{ a.s.},\\ e(\u,\y(\omega),\omega) &= u_e(\omega) \text{ a.s.}, \\
  i(\u, \y(\omega), \omega) &\leq_K u_i(\omega) \text{ a.s.}
  \end{aligned}\right.
   \end{aligned}
\end{equation}
where $\varphi(x,u) = j(x)$ if all constraints
of~\eqref{eq:model-PDE-UQ-state-constraints-perturbed} are fulfilled, and
$\varphi(x,u) = \infty$ otherwise. 
We define the space of perturbations by
\begin{equation*}
  U := L^\infty(\Omega,W) \times L^\infty(\Omega,R)
\end{equation*}
and the space of Lagrange multipliers by
\begin{equation*}
  \Z:=L^1(\Omega,W^*) \times L^1(\Omega,R^*).
\end{equation*}  
These spaces can be paired for $u=(u_e,u_i) \in U$ and
$\lambda = (\p,\lami) \in \Lambda$ with the bilinear form
\begin{equation}
\label{eq:dual-pairing}
\langle u,\z \rangle_{U,\Z}:=\int_{\Omega} \langle u_e(\omega), \p(\omega) \rangle_{W,W^*} + \langle u_{i}(\omega), \lami(\omega) \rangle_{R,R^*} \D \pP(\omega).
\end{equation}
The generalized Lagrangian on $X \times \Z$ is defined by
\begin{equation}
\label{eq:generalized-Lagrangian-definition}
L(x,\lambda):=\inf_{u \in U} \left\lbrace \langle u,\lambda \rangle_{U,\Z} + \varphi(x,u) \right\rbrace.
\end{equation}
Given the sets
\begin{align*}
X_0&:=\{ x=(\u,\y)\in X: \u \in C_1 \text{ and } \y(\omega) \in C_2 \text{ a.s.}\},\\
\Z_0 &:=\{ \lambda=(\p,\lami) \in \Z: \lami(\omega) \in K^\oplus \text{ a.s.}\},
\end{align*}
it is possible to show (see Appendix) that the Lagrangian takes the form
\begin{equation}
\label{eq:Lagrangian}
L(x,\z)= \begin{cases}
J_1(\u) +\E[\bar{J}_2(\u,\y(\cdot), \z(\cdot), \cdot)], & \text{ if } x \in X_0, 
\z \in \Z_0 \\
-\infty, & \text{ if } x \in X_0, \z \not\in \Z_0,\\
\phantom{-}\infty, & \text{ if } x \not\in X_0,
\end{cases}
\end{equation}
where $\bar{J}_2(\u,\y,\z,\omega) :=J_2(\u,\y)+\langle \p, e(\u,\y,\omega) \rangle_{W^*,W} +\langle \lami, i(\u,\y,\omega) \rangle_{R^*,R}.$
A saddle point of $L$ is by definition a point $(\bar{x},\bar{\z}) \in X \times \Z$ such that 
\begin{equation}
\label{eq:saddle-point-property}
L(\bar{x},\z) \leq L(\bar{x},\bar{\z}) \leq L(x,\bar{\z}) \quad \forall (x,\z) \in X\times \Z.
\end{equation}

Now, we define the dual problem
\begin{equation}\tag{$\textup{D}$}
\label{eq:dual-problem}
\max_{\z \in \Z} \left\lbrace g(\z) := \inf_{x \in X} L(x,\z) \right\rbrace.
\end{equation}
By basic duality, the question of the existence of saddle points is the
same as identifying those $(\bx,\bz)$ for which the minimum of
Problem~\eqref{eq:model-problem-abstract-long} and maximum of
Problem~\eqref{eq:dual-problem} is attained, i.e.,
\begin{equation*}
  \inf \textup{P} = \inf_{x \in X} \sup_{\z \in \Z} L(x,\z) = \sup_{\z \in \Z} \inf_{x \in X} L(x,\z) = \sup \textup{D}.
\end{equation*}
By the above definitions, it is clear that for all $x \in X_0$, $j(x) = \sup_{\z \in \Z} L(x,\z)$
and $\varphi(x,0) = j(x)$, from which we get
\begin{align*}
\varphi(x,u) & = \sup_{\z \in \Z_0} \{ L(x,\z) - \langle u,\z\rangle_{U,\Z}\}.
\end{align*}
It is straightforward to show that $L$ is convex in $x$ for given $\lambda \in \Lambda_0$ and concave in
$\z$ and that $\varphi$ is convex in $(x,u)$.
Moreover,
$\varphi \not\equiv \infty$. %(this can be seen by taking $x_0 \in X_0$ and setting $u_e =  e(\u,\y,\omega)$ and $v_{\lambda} = i(\u,\y,\omega)$).
It will be convenient to define $X' = \U^* \times L^1(\Omega,\Y^*)$ and
the pairing
\begin{equation}
\label{eq:pairing-on-X}
\langle x,x'\rangle_{X,X'}=\langle \u, \u' \rangle_{\U,\U^*} + \int_{\Omega} \langle \y(\omega),\y'(\omega))\rangle_{\Y,\Y^*} \D \pP(\omega). 
\end{equation}
\begin{lemma}
\label{lemma:lsc-conjugate-function}
Let  \cref{assumption:general-problem} be satisfied. Then the
function $\varphi: X \times U \rightarrow \R \cup \{ \infty\}$ is
weak$^*$ lower semicontinuous.
\end{lemma}

\begin{proof}
We argue as in~\cite[Proposition~3]{Rockafellar1976}. Let $Y:=X \times U$ and denote the pairing on $Z:=X' \times \Z$ by
\begin{equation}
\label{eq:product-dual-pairing}
\langle y,z\rangle_{Y,Z}:= \langle x,x'\rangle_{X,X'} + \langle u,\lambda \rangle_{U,\Z}.
\end{equation}  
Since $Y=Z^*$, the topology induced by the pairing \eqref{eq:product-dual-pairing} coincides with the weak$^*$ topology on $Y$. We define
  $\varphi_1(\u) = J_1(\u),$ if $\u \in C_1$ and $\varphi_1(\u) = \infty$
  if $\u \not\in C_1$ and
\begin{align*}
&\varphi_2(\u,\y,u,\omega) = \begin{cases}
J_2(\u,\y), & \text{if } \y \in C_2, e(\u,\y,\omega) =u_e, \, i(\u, \y, \omega) \leq_K u_i,\\
\phantom{-}\infty, & \text{otherwise}.
\end{cases}
\end{align*}
Obviously, $\varphi(x,u) =\varphi_1(\u) + \int_{\Omega} \varphi_2(\u,\y(\omega),u(\omega),\omega) \D \pP(\omega)$.
Let $\langle \cdot,\cdot \rangle_{Y',Z'}$ denote the pairing of
$Y':=\U \times X_2 \times (W\times R)$ with
$Z':=\U^* \times \Y^* \times (W^* \times R^*)$; then the conjugate
integrand to $\varphi_2$ is given by
\begin{equation*}
  \varphi_2^*(z',\omega) = \sup_{y' \in Y'} \{ \langle y',z' \rangle_{Y',Z'} - \varphi_2(y',\omega)\}.
\end{equation*}

Defining $h(y',\omega)= J_2(\u,\y)$ for $y'= (\u,\y,u)$ we have
$h(y',\omega) \leq \varphi_2(y',\omega)$ a.s. The function $h$ is a normal
convex integrand and is integrable on
$X_1 \times L^\infty(\Omega,X_2) \times (L^\infty(\Omega,W) \times L^\infty(\Omega,R))$
by  \cref{assumption:general-problem}. Thus with the conjugate
integrand $h^*$, $I_h$ and $I_{h^*}$ are conjugate to each other by
\cref{cor:duality-conjugate-functions}, meaning that
$I_{h^*} \not\equiv \infty.$ 

Since $h \leq \varphi_2$ we have $h^* \geq \varphi_2^*$,
and hence there exists a point
$z \in Z$ such that $I_{\varphi_2^*}(z) < \infty$.
Since there clearly exists a point such that $I_{\varphi_2}$ is finite, it
follows that $I_{\varphi_2}$ and $I_{\varphi_2^*}$ are conjugate to one another
and are weak$^*$ lower semicontinuous,
%with respect to the pairing~\eqref{eq:product-dual-pairing}
see~\cite[p.~227]{Rockafellar1971}.
Since $\varphi_1$ is also weakly lower semicontinuous with respect to the
natural pairing on the reflexive space $X_1$, $\varphi_1$
and hence $\varphi$ are also weak$^*$
lower semicontinuous. 
\end{proof}

The following result is based on~\cite[Theorem 3]{Rockafellar1976}. We
define the value function 
\begin{equation}
\label{eq:value-function}
v(u): = \inf_{x \in X} \varphi(x,u).
\end{equation}
Obviously, $v(0) = \inf \textup{P}$. For the next result, we define the
second-stage admissible set by
\begin{equation}
\label{eq:second-stage-feasible-set}
X_{2,0} = \{ \y \in L^\infty(\Omega,\Y): \y(\omega) \in C_2 \text{ a.s.}\}.
\end{equation}

\begin{theorem}
\label{thm:minP-supD}
Let  \cref{assumption:general-problem} be satisfied. Supposing
$C_1$ and $C_2$ are bounded sets, then 
\begin{equation*}
  -\infty < \min \textup{P} = \sup \textup{D},
\end{equation*}
meaning that the primal problem attains its minimum, and the minimal
value coincides with the supremum of the dual, which need not be attained.
\end{theorem}

\begin{proof}
%Let $\mathcal{T}_2$ denote the weak topology in $L^\infty(\Omega,\Y)$
%induced by $L^1(\Omega,\Y^*)$.
We first show that $X_{2,0}$ is compact with respect to the weak$^*$ topology on $L^\infty(\Omega,\Y)$.
This follows
by showing that $I_h$ and $I_{h^*}$ are conjugate to each other, where
$h(\y,\omega) := \delta_{C_2}(x_2)$ and $h^*$ denotes the conjugate of $h$. Since $C_2 \neq \emptyset$ is
convex and closed, $h$ is a normal convex integrand. It is easy to see that
$h^*(0,\omega) = 0$, so in particular $I_{h^*}(0) < \infty$, meaning there
exists a point where $I_{h^*}$ is finite. Note $I_h$ is also finite in at
least one point since $C_2$ is nonempty. It follows that $I_h$ and $I_{h^*}$
are conjugate to one another, meaning that $I_h$ is lower semicontinuous 
with respect to the weak$^*$ topology on $L^\infty(\Omega,\Y)$. 
In particular, for a weak$^*$ convergent sequence $\{y_n\} \subset X'_{2,0}:= \{ \y \in L^\infty(\Omega,\Y): I_h(\y) \leq 0 \}$ such that $y_n \rightharpoonup^* \bar{y}$ it follows that 
\begin{align*}
\liminf_{n \rightarrow \infty} I_{h}(y_n) \geq   I_{h}(\bar{y}),
\end{align*}
so $\bar{y} \in X'_{2,0}$; hence, $X'_{2,0}$ is closed with respect to
to the weak$^*$ topology. Here, we used the fact that weak* compactness coincides with weak* sequential compactness on $L^\infty(\Omega,X_2)$, since it is the dual of a separable space.
By definition of $h$, we deduce that
$\bar{y}(\omega) \in C_2$ a.s.~and therefore $X_{2,0}$ is also closed. Of
course, $X_{2,0}$ is bounded, so $X_{2,0}$ is weak$^*$ compact, see,
e.g.,~\cite[Corollary V.4.3]{DunfordSchwartz:I:1957}.
It is clear that the set $C_1$ is compact in $\U$ with
respect to the weak topology on $\U$. It therefore follows
that $X_0$ is weak$^*$ compact.

Since $X_0$ is weak$^*$ compact and by \cref{lemma:lsc-conjugate-function},
$\varphi$ is weak$^*$ lower semicontinuous on $X\times U$, we have
for all $u \in U$ that
\begin{equation*}
  \inf_{x \in X} \varphi(x,u) = \inf_{x \in X_0} \varphi(x,u) = \min_{x \in X_0} \varphi(x,u) = v(u) > -\infty.
\end{equation*}
It is easy to verify $-v^*(-\z) = g(\z)$ and hence
$v^{**}(u) = \sup_{\z \in \Z} \{ g(\z) - \langle \z,u \rangle_{\Z,U}\}$.
It follows that
\begin{equation*}
  v^{**}(0) = \sup_{\z \in \Z} g(\z) = \sup \textup{D}.
\end{equation*}
To conclude the proof, we show that $v$ is weak$^*$ lower semicontinuous in $U$.
Notice that the level set 
$ \textup{lev}_\alpha \varphi = \{ (x,u) \in X \times U: \varphi(x,u) \leq \alpha\}$
is weak$^*$-closed by weak$^*$ lower semicontinuity of $\varphi$,
see \cref{lemma:lsc-conjugate-function}.
Additionally, $\varphi$ is
finite only if $x \in X_0$, so the projection of
$\textup{lev}_\alpha \varphi$
onto $X$ is contained in $X_0$. Thus the projection of
$\textup{lev}_\alpha \varphi$ onto $U$, which corresponds to the level set
$\{ u \in U: v(u) \leq \alpha \}$, is closed in the weak$^*$ topology, from
which we conclude that $v$ is weak$^*$ and weak lower semicontinuous.
Since $v > -\infty$ and $v$ is convex and lower semicontinuous, we have
that $v^{**} = v$ (cf.~\cite[Theorem 2.113]{Bonnans2013}) and therefore 
\begin{equation*}
  -\infty<\min \textup{P} = v(0) = v^{**}(0) = \sup \textup{D}.
\end{equation*}
\end{proof}

\begin{corollary}
Let  \cref{assumption:general-problem} be satisfied and $j$ be
radially
unbounded, i.e., $j(x) \rightarrow \infty$ as $\|x\|_X \rightarrow \infty$
then 
\begin{equation*}
  - \infty <\min \textup{P} = \sup \textup{D},
\end{equation*}
meaning that the primal problem attains its minimum, and the minimal
value coincides with the supremum of the dual, which need not be attained.
\end{corollary}
\begin{proof}
Inspection of the proof of \cref{thm:minP-supD} shows that the only
place where boundedness of
$C_1$ and $C_2$ comes into play is the weak$^*$ compactness of $X_0$. However,
if $x_0\in X$ is an arbitrary
feasible point of~\eqref{eq:model-PDE-UQ-state-constraints-perturbed} then
the set $N_0:=\{ x\in X \colon j(x) \le j(x_0)\}$ is bounded due to radial
unboundedness of $j$. Hence, clearly,
\begin{equation*}
  \inf_{x \in X} \varphi(x,u) = \inf_{x \in X_0\cap N_0} \varphi(x,u) = \min_{x \in X_0\cap N_0} \varphi(x,u) = v(u) > -\infty
\end{equation*}
holds and the proof of \cref{thm:minP-supD} can be repeated.
\end{proof}

\cref{thm:minP-supD} has shown that a necessary condition for the
minimum to be obtained in Problem~\eqref{eq:model-problem-abstract-long}
is for $C_1$ and $C_2$ to be bounded sets. We will now focus on establishing
sufficient conditions. Recalling \cref{def:singular-functionals},
let $\mathcal{S}_e$ and $\mathcal{S}_i$ denote the sets of singular
functionals defined on $L^\infty(\Omega,W)$ and $L^\infty(\Omega,R)$,
respectively. 
We define
\begin{align*}
\Z^\circ &= \{ \z^\circ = (\p^\circ,\lami^\circ) \in \mathcal{S}_e \times \mathcal{S}_i\},\\
\Z_0^\circ &= \{ \z^\circ = (\p^\circ,\lami^\circ) \in \Z^\circ:  \lami^\circ(y) \geq 0\, \forall y \in L^\infty(\Omega,R): y \geq_K 0\text{ a.s.} \},
\end{align*}
as well as $ L^\circ(x,\z^\circ)=\p^\circ(e(\u,\y(\cdot),\cdot)) +\lami^\circ(i(\u,\y(\cdot),\cdot)).$
Given $\lambda^\circ \in \Lambda_0^\circ$, notice 
\begin{equation}
\label{eq:feasibility-negative-Lagrangian-term}
e(\u,\y(\omega),\omega) = 0, i(\u, \y(\omega),\omega) \leq_K 0 \text{ a.s.} \Rightarrow L^\circ(x,\lambda^\circ)\leq 0.
\end{equation}
Also, from the results in \cref{subsec:subdifferentiability-Linf},
we have that
$(\p,\p^\circ) \in L^1(\Omega,W^*) \times \mathcal{S}_e \cong(L^\infty(\Omega,W))^*$
and
$(\lami,\lami^\circ) \in L^1(\Omega,R^*) \times \mathcal{S}_i \cong(L^\infty(\Omega,R))^*$.
This means that $\Z \times \Z^\circ$ characterizes the dual space
$(L^\infty(\Omega,W) \times L^\infty(\Omega,R))^*$. Here, we are interested
in finding conditions under which the singular part $\Z^\circ$ vanishes in
the optimum.

With that goal in mind, we define an extension of the
Lagrangian~\eqref{eq:Lagrangian} for
Problem~\eqref{eq:model-problem-abstract-long} on the space
$X \times \Z \times \Z^\circ$ via
\begin{equation}
\label{eq:Lagrangian-extended}
\bar{L}(x,\z,\z^\circ)= \begin{cases}
L(x,\z) + L^\circ(x,\z^\circ) & \text{ if } x \in X_0, (\z,\z^\circ) \in \Z_0\times \Z_0^\circ, \\
-\infty, & \text{ if } x \in X_0, (\z,\z^\circ) \not\in \Z_0\times \Z_0^\circ,\\
\phantom{-}\infty, & \text{ if } x \not\in X_0.
\end{cases}\hspace*{-4mm}
\end{equation}
The corresponding extended dual problem is given by
\begin{equation}\tag{$\text{$\bar{\textup{D}}$}$}
\label{eq:dual-tilde-problem}
\max_{(\z,\z^\circ) \in \Z \times \Z^\circ} \left\lbrace \bar{g}(\z,\z^\circ) := \inf_{x \in X} \bar{L}(x,\z,\z^\circ) \right\rbrace.
\end{equation}
Clearly, $\bar{g}(\z,0) = g(\z)$ and thus
$\sup \textup{D} \leq \sup \bar{\textup{D}}.$
Additionally, $ \sup \bar{\textup{D}} \leq \inf \textup{P}$, since
by~\eqref{eq:feasibility-negative-Lagrangian-term}, we have
\begin{align*}
\sup_{(\z,\z^\circ)} \bar{g}(\z,\z^\circ) &= \sup_{(\z,\z^\circ)} \inf_{x \in X} \{ L(x,\z) + L^\circ(x,\z^\circ)\} \leq \inf_{x \in X} \sup_{(\z,\z^\circ)}  \{ L(x,\z) + L^\circ(x,\z^\circ)\}. 
\end{align*}

For a sufficient condition, we introduce the induced feasible set for the
first-stage variable $\u$:
\begin{align*}
\tilde{C}_1 := \{ \u \in \U\,&:\, \exists \y \in L^\infty(\Omega,\Y) \text{ s.t. } e(\u,\y(\omega),\omega) = 0 \text{ a.s.}, \\
&\qquad \quad  i(\u, \y(\omega), \omega) \leq_K 0 \text{ a.s.}, \, \y(\omega) \in C_2 \text{ a.s.}\}
\end{align*}
Problem~\eqref{eq:model-problem-abstract-long} is said to satisfy
the \textit{relatively complete recourse} condition if and only if 
\begin{equation}
\label{eq:relatively-complete-recourse}
C_1 \subset \tilde{C}_1.
\end{equation}
\begin{remark}
In fact, it is possible to relax this assumption to
$\text{ri } C_1 \subset \tilde{C}_1^{\circ}$, where $\text{ri } C_1$ denotes
the relative interior of $C_1$ and $\tilde{C}_1^{\circ}$ represents the
\textit{singularly induced feasible set}; see~\cite{Rockafellar1976b} for
more details. 
\end{remark}

Additionally, we will require a regularity condition. We call the problem
\textit{strictly feasible} if the value function $v$, defined
in~\eqref{eq:value-function}, satisfies
\begin{equation}
\label{eq:constraint-qualification}
0 \in \textup{int}\,\textup{dom}\, v.
\end{equation}

\begin{remark}
\label{rem:Slater-condition-explanation}
The condition~\eqref{eq:constraint-qualification} implies
by~\cite[Theorem 18]{Rockafellar1974} that $v$ is bounded above in a
neighborhood of zero and is continuous at zero. Notice that
$v(u) = \inf_{x \in X} \varphi(x,u)$ is only finite (and equal to $j(x)$)
if the constraints are satisfied, meaning $\u \in C_1$ and almost surely
$\y(\omega) \in C_2, e(\u,\y(\omega),\omega) = u_e(\omega), i(\u, \y(\omega), \omega) \leq_K u_i(\omega)$.
This condition can therefore be thought of as an ``almost sure''
Slater condition. The condition induces an interplay between the spaces $W$ and $R$. Additionally, since $i(x_1,x_2(\omega),\omega) \leq_K u_i(\omega)$ needs to be satisfied in a neighborhood of zero in $R$,  this in general implicitly requires that $K$ has interior points.
\end{remark}

\begin{theorem}
\label{thm:infP-maxD}
Let  \cref{assumption:general-problem} be satisfied. Suppose the
relatively complete recourse
condition~\eqref{eq:relatively-complete-recourse} is satisfied and
Problem~\eqref{eq:model-problem-abstract-long} is strictly feasible,
i.e.,~\eqref{eq:constraint-qualification} holds. Then 
\begin{equation*}
  \inf{\textup{P}}=\max{\textup{D}} < \infty,
\end{equation*}
meaning that the dual problem attains its maximum, and the maximal
value coincides with the infimum of the primal, which need not be attained.
\end{theorem}
\begin{proof}
We modify the arguments from~\cite[Theorem 3]{Rockafellar1976a} to fit our
setting. By  \cref{rem:Slater-condition-explanation}, $v$ is bounded
above on a neighborhood of zero, so we have
by~\cite[Theorem 17]{Rockafellar1974} that
\begin{equation}
\label{eq:infP-equal-maxD-tilde}
\inf \textup{P} = \max \bar{\textup{D}} < \infty.
\end{equation}

In the next step, we prove that
condition~\eqref{eq:relatively-complete-recourse} implies 
\begin{equation}
\label{eq:relation-dual-extended-dual}
\bar{g}(\z,\z^\circ) \leq  g(\z) \quad \forall (\lambda,\lambda^\circ) \in \Lambda_0 \times \Lambda_0^\circ.
\end{equation}
With this the proof will be complete since now,
$ \max \bar{\textup{D}} \le \sup \textup{D} \le \max \bar{\textup{D}}$
is asserted and a solution $(\lambda,\lambda^\circ)$
of~\eqref{eq:dual-tilde-problem} gives a solution
$\lambda$ of~\eqref{eq:dual-problem}.

To show~\eqref{eq:relation-dual-extended-dual}, let $(\z,\z^\circ) \in \Z_0 \times \Z_0^\circ$ be arbitrary.
Recalling the feasible set~\eqref{eq:second-stage-feasible-set}, we define
\begin{equation*}
  \ell(\u,\z^\circ)=\inf_{\y \in X_{2,0}} L^\circ(x,\z^\circ).
\end{equation*}
We skip the trivial case $\bar{g}(\z,\z^\circ) = -\infty$ and now
show that
\begin{equation}
\label{eq:proof-infP-maxD-q}
\bar{g}(\z,\z^\circ) =\inf_{x \in X_0} \{ L(x,\z) + \ell(\u,\z^\circ)\}.
\end{equation}
It is obvious that
\begin{align*}
  &\inf_{\y \in X_{2,0}} \left\lbrace \E[ \bar{J}_2(\u,\y(\cdot),\z(\cdot),\cdot)]
  + L^\circ(x,\z^\circ)\right\rbrace\\
  &\quad \geq \inf_{\y \in X_{2,0}} \E[ \bar{J}_2(\u,\y(\cdot),\z(\cdot),\cdot) ]
  + \inf_{\y \in X_{2,0}} L^\circ(x,\z^\circ).
\end{align*}
By definition, for the functional $\p^\circ$ there exists a decreasing
sequence of sets $\{ F_{e,n}\} \subset \mathcal{F}$  such that
$\pP(F_{e,n}) \rightarrow 0$ as $n \rightarrow \infty$ and
$\p^\circ(w) = 0$ for all $w \in L^\infty(\Omega,W)$ such that $w = 0$
a.s.~on $F_{e,n}$. The sets $F_{i,n}$ corresponding to $\lami^\circ$ are
defined analogously.
We define $F_n = F_{e,n}\cup F_{i,n}$ and
\begin{equation*}
  y_n(\omega) = \begin{cases}
y'(\omega), &\omega \in F_n\\
y''(\omega), &\omega \not\in F_n
  \end{cases}
\end{equation*}
for arbitrary $y', y'' \in X_{2,0}$.
If $\omega \in F_n$, then
$e(\u,y_n(\omega),\omega) = e(\u,y'(\omega),\omega)$ and
$i(\u,y_n(\omega),\omega) = i(\u,y'(\omega),\omega),$ implying
$  \p^\circ(e(\u,y_n(\omega),\omega)) = \p^\circ(e(\u,y'(\omega),\omega))$
and
$  \lami^\circ (i(\u,y_n(\omega),\omega)) = \lami^\circ(i(\u,y'(\omega),\omega)).$
Thus, for any $y', y''$, and $\varepsilon >0$, there exists an $n_0$ such
that for $n\geq n_0$ and $x_2=y_n$ it holds that
\begin{align*}
 &\E[ \bar{J}_2(\u,\y(\cdot),\z(\cdot),\omega)] +\p^\circ(e(\u,\y(\cdot),\cdot))  + \lami^\circ (i(\u,\y(\cdot),\cdot))\\
 &\quad \leq \E[\bar{J}_2(\u,y''(\cdot),\z(\cdot),\cdot)] + \p^\circ(e(\u,y'(\cdot),\cdot))  + \lami^\circ (i(\u,y'(\cdot),\cdot)) + \varepsilon.
\end{align*}
With that, we have shown~\eqref{eq:proof-infP-maxD-q}. We now define
\begin{equation*}
  h(\u) = \begin{cases}
\inf_{\y \in X_{2,0}} L(x,\z), & \text{if } \u \in C_1, \\
 \infty, & \text{else}
  \end{cases} \qquad \text{and} \qquad   k(\u) = - \ell(\u,\z^\circ).
\end{equation*}
Notice that $\bar{g}(\z,\z^\circ) = \inf_{\u \in \U} \{ h(\u) - k(\u)\}$.
Additionally, $h \not\equiv \infty$ is convex and $k>-\infty$ is concave.
Since $\bar{g}$ is finite, $k\not\equiv \infty$ and $h$ must be proper. Therefore, with
$h^*(v) = \sup_{\u \in \U} \{ \langle v, \u \rangle_{\U^*,\U} - h(\u)\}$ and
$k^*(v) = \inf_{\u \in \U} \{ \langle v,\u \rangle_{\U^*,\U} - k(\u)\}$, we have by Fenchel's duality theorem
(cf.~\cite[Theorem 6.5.6]{Aubin1990}) that
\begin{equation}
\label{eq:dual-max-equality}
\bar{g}(\z,\z^\circ) = \max_{\u^* \in \U^*} \{ k^*(\u^*) - h^*(\u^*)\}.
\end{equation}
Let $\u^*$ denote the maximizer of~\eqref{eq:dual-max-equality}, meaning
$\bar{g}(\z,\z^\circ) = k^*(\u^*) - h^*(\u^*).$
Then by definition of $h^*$, we have for all $\u \in \U$ that
\begin{equation}
\label{eq:bound-h-star}
h(\u) - \langle \u^*,\u\rangle_{\U^*,\U} \geq \bar{g}(\z,\z^\circ) - k^*(\u^*).
\end{equation}
Likewise by definition of $k$ and $k^*$, we get
\begin{equation*}
  \ell(\u,\z^\circ) + \langle \u^*,\u\rangle_{\U^*,\U} \geq k^*(\u^*).
\end{equation*}
It is clear that $\ell(\u,\z^\circ) \leq 0$ for all
$\u \in \tilde{C}_1$. Indeed, $\u \in \tilde{C}_1$ implies that there
exists a $x_2 \in X_{2,0}$ satisfying $e(\u,\y(\omega),\omega)=0$ and
$i(\u,\y(\omega),\omega)\leq_K 0$ a.s.
Recalling~\eqref{eq:feasibility-negative-Lagrangian-term}, we get $ \langle \u^*,\u \rangle_{\U^*,\U} \geq k^*(\u^*)$ for all $\u \in \tilde{C}_1 \supset C_1$.
From~\eqref{eq:bound-h-star} we thus have for all $\u \in C_1$ that 
$h(\u) \geq \bar{g}(\z,\z^\circ)$ holds, and hence 
\begin{equation*}
  L(x,\z) \geq h(\u) \geq \bar{g}(\z,\z^\circ)
\end{equation*} 
for all $x \in X_0$ and all
$(\lambda,\lambda^\circ) \in \Lambda\times \Lambda^\circ$.
It follows that $g(\z) \geq \inf_{x \in X_0} L(x,\z) \geq \bar{g}(\z,\z^\circ)$
and we have shown~\eqref{eq:relation-dual-extended-dual} finishing the proof.
\end{proof}

\begin{remark}
\label{rem:discrete-probability-space}
If the probability space is finite in the sense that $\Omega$ contains a finite number of points, then $L^\infty(\Omega,X)$ is reflexive since $X$ is reflexive; see~\cite[p.~100, Corollary~2]{Diestel1977}. In particular, $L^1(\Omega,X^*)$ and $L^\infty(\Omega,X)$ are paired spaces with the weak topology, and the Lagrangian $L(x,\lambda)$ coincides with the extended Lagrangian $L(x,\lambda,\lambda^\circ)$. Hence $L^\circ(x,\lambda^\circ) \equiv 0$ and \cref{thm:infP-maxD} holds without the relatively complete recourse condition \eqref{eq:relatively-complete-recourse}. This property can be exploited to obtain regular Lagrange multipliers for methods relying on a discrete approximation of (otherwise continuous) sample space $\Omega$.
\end{remark}

\subsection{Karush--Kuhn--Tucker Conditions}
\label{subsec:KKT}
In \cref{sec:existence-saddle-points}, we showed that saddle points
of the generalized Lagrangian exist under relatively mild assumptions. We
require that the constraint sets $C_1$ and $C_2$ are bounded. Additionally,
the problem must satisfy an almost sure strict feasibility condition in
addition to a standard assumption in stochastic models known as a relatively complete recourse assumption. We now turn to obtaining optimality conditions under
the assumption that a saddle point exists. This leads us to the following
central result.
\begin{theorem}
\label{thm:KKT-basic-conditions}
Let  \cref{assumption:general-problem} be satisfied. Then 
$(\bar{x},\bar{\z}) \in (X_1 \times L^\infty(\Omega,X_2))$ $\times (L^1(\Omega,W^*) \times L^1(\Omega,R^*))$
is a saddle point of the Lagrangian~\eqref{eq:Lagrangian} if and only if
there exists a function $\rho \in L^1(\Omega,\U^*)$ such that the following conditions are satisfied:
\begin{enumerate}[label=(\roman*)]
\item The function
  \begin{equation*}
    \u \mapsto J_1(\u) + \langle \E[\rho], \u \rangle_{\U^*,\U}
  \end{equation*}
attains its
minimum over $C_1$ at $\bu$. \label{eq:FOC-basic-problem1}
\item The function 
\begin{align*}(\u,\y) \mapsto &J_2(\u,\y) + \langle \bp(\omega), e(\u,\y,\omega)\rangle_{W^*,W} \\
&\quad + \langle \bl(\omega),i(\u,\y,\omega) \rangle_{R^*,R} - \langle \rho(\omega), \u \rangle_{\U^*,\U} 
\end{align*}
attains its minimum in $\U \times C_2$ at $(\bu,\by(\omega))$ for almost
every $\omega \in \Omega$. \label{eq:FOC-basic-problem2}
\item It holds that $\bu \in C_1$ and the following conditions hold almost
  surely:
\begin{align*}
&e(\bu, \by(\omega),\omega)=0, \quad \by(\omega) \in C_2,\quad \bl(\omega)\in K^\oplus,\\
& i(\bu, \by(\omega),\omega) \leq_K 0, \quad \langle \bl(\omega), i(\bu, \by(\omega),\omega) \rangle_{R^*,R} = 0.
\end{align*}
\label{eq:FOC-basic-problem3}
\end{enumerate}
\end{theorem}

The appearance of this extra Lagrange multiplier $\rho$ in
\cref{thm:KKT-basic-conditions} might seem surprising; however, it
is standard in two-stage stochastic optimization. It is known as
a ``nonanticipativity'' constraint and comes from this particular setting,
where the first stage variable $\u$ is deterministic and the second-stage
variable $\y$ is random.

\paragraph{Proof of \cref{thm:KKT-basic-conditions}}
We follow the arguments from~\cite[Section 3]{Rockafellar1975}. We first
show that the existence of a saddle point implies condition (iii). Notice
that $(\bx,\blam)$ can only be a saddle point if
$(\bar{x},\bar{\lambda}) \in X_0 \times \Lambda_0$, which immediately
implies
\begin{equation*}
  \bu \in C_1, \quad \by(\omega) \in C_2 \text{ a.s.},\quad \bl(\omega) \in K^{\oplus} \text{ a.s.}
\end{equation*}
For $\bar{x} = (\bu,\by)$, we have by definition of the
Lagrangian~\eqref{eq:Lagrangian} that
\begin{align*}
&\sup_{\z \in \Z_{0}} L(\bar{x},\z) = \sup_{\z\in \Z_{0}} \Big\lbrace J_1(\bu) + \int_{\Omega} \bar{J}_2(\bu,\by(\omega), \z(\omega), \omega) \D \pP(\omega) \Big\rbrace.
\end{align*}
We now show that $\sup_{\z \in \Z_{0}} L(\bar{x},\z) = \infty$ unless
$e(\bu,\by(\omega),\omega)=0$ and $ i(\bu, \by(\omega),\omega))\leq_K 0$
a.s. Indeed, suppose that the set $E:=\{ \omega \in \Omega: -i(\bu, \by(\omega),\omega)\not\in K\}$
has positive probability, meaning $\pP(E) > 0.$ Then defining
$\lambda_n\equiv n$ on $E$ and $\lambda_n\equiv 0$ on
$\Omega \backslash E$, one gets
$\E[\langle \lambda_n, i(\u,\y(\cdot),\cdot) \rangle_{R^*,R}] \rightarrow \infty$
as $n \rightarrow \infty$. An analogous argument can be applied to the
equality constraint. Now, since $\bar{\lambda}_i(\omega) \in K^\oplus$ and
$ i(\bu, \by(\omega),\omega)\leq_K 0$ a.s., we have that
$\langle\bar{\lambda}_i(\omega), i(\bu, \by(\omega),\omega)\rangle_{R^*,R} \leq 0$
a.s. The supremum of $L(\bx, \z)$ can therefore only be attained at
$\blam$ if and only if
$\langle \bar{\lambda}_i(\omega), i(\bu, \by(\omega),\omega) \rangle_{R^*,R} = 0$
a.s. We have shown that if $(\bx,\blam)$ is a saddle point, then condition
(iii) is fulfilled.

It is easy to see that conditions (i)--(iii) imply that $(\bx,\bz)$ is a
saddle point. Indeed, for every $x = (\u,\y) \in X$, conditions (i)--(ii)
imply
% \footnote{Hille's theorem (cf.~\cite[Theorem 1.2.4]{Hytonen2016}) provides justification for ``exchanging'' the dual pairing and the expectation, meaning if $f:\Omega \rightarrow X$ is integrable, then for all $x^* \in X^*$ we have
%$\langle x^*, \E[f]\rangle_{X^*,X} = \E[\langle x^*,f \rangle_{X^*,X}].$}
\begin{align*}
L(\bx,\bz)&= J_1(\bu) + \langle \E[\rho], \bu \rangle_{\U^*,\U} + \E[\bar{J}_2(\bu,\by(\cdot),\bz(\cdot),\cdot) - \langle \rho(\cdot),\bu \rangle_{\U^*,\U}]\\
& \leq J_1(\u) + \langle \E[\rho], \u \rangle_{\U^*,\U} + \E[\bar{J}_2(\u,\y(\cdot),\bz(\cdot),\cdot) - \langle \rho(\cdot),\u\rangle_{\U^*,\U}] \\
&= L(x,\bz).
\end{align*}
To show that $L(\bar{x},\lambda) \leq L(\bar{x},\bar{\lambda})$ for all
$\lambda \in \Lambda$, it is enough to show that 
\begin{equation} 
\label{eq:expecation-inequality-KKT-proof}
\E[\bar{J}_2(\bu,\by(\cdot),\z(\cdot),\cdot)] \leq  \E[\bar{J}_2(\bu,\by(\cdot),\bz(\cdot),\cdot)] \quad \forall \lambda \in \Lambda.
\end{equation}
Since $e(\bu,\by(\omega),\omega) = 0$ and
$\langle \lambda_i(\omega),i(\bu,\by(\omega),\omega)\rangle_{R^*,R} \leq 0$
a.s.,~\eqref{eq:expecation-inequality-KKT-proof} must certainly be
satisfied, since (as we argued before) the maximum of $L(\bar{x}, \lambda)$
can only be attained if
$\langle \bar{\lambda}_i(\omega),i(\bu,\by(\omega),\omega)\rangle_{R^*,R}=0$
a.s.

Now, for the most involved part of the proof, we show that if $(\bx,\bz)$
is a saddle point, then conditions (i) and (ii) must be satisfied. To
simplify, we redefine $\bl$ so that $\bl(\omega) \geq 0$ for \textit{all}
$\omega \in \Omega$. We define 
\begin{equation}
\begin{aligned}
\label{eq:def-h2}
&h_2(\u,\y,\omega)\\
&=J_2(\u,\y) + \langle \bp(\omega), e(\u,\y,\omega)\rangle_{W^*,W} + \langle \bl(\omega), i(\u,\y,\omega)\rangle_{R^*,R}.
\end{aligned}
\end{equation}
The function $h_2$ is clearly convex in $X$;
$h_2(\u(\omega),\y(\omega),\omega)$ is integrable by
 \cref{assumption:general-problem} and the fact that
$\bp \in L^1(\Omega,W^*)$ and $\bl \in L^1(\Omega,R^*)$.
In particular, we get by \cref{cor:duality-conjugate-functions}
that 
\begin{equation*}
  H_2(\u,\y): = \int_\Omega h_2(\u(\omega),\y(\omega),\omega) \D \pP(\omega)
\end{equation*}
is well-defined and finite on $L^\infty(\Omega,\U) \times L^\infty(\Omega,\Y)$
as well as convex and continuous. 

Let $\iota: \U \times L^\infty(\Omega,\Y) \rightarrow L^\infty(\Omega,\U) \times L^\infty(\Omega,\Y)$
be the continuous injection, which maps elements of $\U$ to the
corresponding constant in $L^\infty(\Omega,\U)$ and maps each element
of $L^\infty(\Omega,\Y)$ to itself. Setting $H_1(x_1,x_2) = J_1(x_1)$ if $x \in X_0$ and $H_1(x_1,x_2) = \infty$ otherwise, we have
\begin{equation*}
  L(x,\bz) = H_1(x_1,x_2) + H_2(\iota(\u,\y)) \quad \forall x \in X_0.
\end{equation*}
From $L(\bx,\bz)=\min_{x \in X_0} L(x,\bz)$ it follows that
\begin{equation*}
  H_1(\bu,\by) + H_2(\iota(\bu,\by)) = \min_{(\u,\y) \in X_0} H_1(\u,\y)+H_2(\iota(\u,\y)).
\end{equation*}
By the Moreau--Rockafellar theorem
(cf., e.g.,~\cite[Theorem 2.168]{Bonnans2013}) we have, where $\iota^*$
maps $(L^\infty(\Omega,\U)\times L^\infty(\Omega,\Y))^*$ to
$(\U \times L^\infty(\Omega,\Y))^*$,
\begin{equation*}
  0 \in \partial H_1(\bu,\by) + \iota^* \partial H_2(\iota(\bu,\by)).
\end{equation*}
In particular, there exists
$q \in (L^\infty(\Omega,\U) \times L^\infty(\Omega,\Y))^*$ such that 
\begin{equation*}
  -\iota^* q \in \partial H_1(\bu,\by) \quad \text{and} \quad q \in \partial H_2(\iota(\bu,\by)).
\end{equation*}

Since $h_2$ satisfies the conditions of
\cref{cor:characterization-subdifferentials}, it follows that
$\partial H_2(\iota(\bu,\by)) \subset (L^\infty(\Omega,\U) \times L^\infty(\Omega,\Y))^*$
consists of continuous linear functionals on
$L^\infty(\Omega,\U) \times L^\infty(\Omega,\Y)$, which can be identified with
pairs $(q_1,q_2) \in L^1(\Omega,\U^*) \times L^1(\Omega,\Y^*)$ such that
\begin{equation}
\label{eq:subdifferential-h2}
q(\omega) = (q_1(\omega),q_2(\omega)) \in \partial h_2(\bu, \by(\omega),\omega) \quad \text{a.s.}
\end{equation}
Notice that for $q_1^* \in L^1(\Omega,\U^*)$, the adjoint
$\iota_1^*:(L^\infty(\Omega,\U))^* \rightarrow \U^*$ satisfies, for any $\u \in \U$,
\begin{align*}
\langle \iota_1^* q_1^*,\u \rangle_{\U^*,\U} & =  \langle q_1, \iota_1 \u \rangle_{L^1(\Omega,\U^*),L^\infty(\Omega,\U)} = \E[ \langle q_1(\cdot), x_1\rangle_{\U^*,\U}]. 
\end{align*}
Hence $\iota^* q = (\E[q_1], q_2) \in \U^* \times L^1(\Omega,\Y^*).$
Thus $-\iota^* q \in \partial H_1(\bu,\by)$ can be written as
\begin{equation*}
  H_1(\u,\y) \geq H_1(\bu,\by) - \langle \E[q_1], \u - \bu\rangle_{\U^*,\U} -\E[\langle q_2, \y - \by\rangle_{\Y^*,\Y}]
\end{equation*}
for all $(\u,\y)\in X$.
Recalling $H_1(\u,\y) = J_1(\u)$ if $x \in X_0$, we get
\begin{equation}
\label{eq:inequality-KKT-proof-0}
J_1(\u) \geq J_1(\bu) - \langle \E[q_1], \u - \bu \rangle_{\U^*,\U} \quad \forall \u \in C_1
\end{equation}
and
\begin{equation}
\label{eq:inequality-KKT-proof-1}
\E[\langle q_2, \y - \by\rangle_{\Y^*,\Y}] \geq 0 \quad \forall \y \in L^\infty(\Omega,\Y): \y(\omega)\in C_2 \text{ a.s.}
\end{equation}
The expression~\eqref{eq:inequality-KKT-proof-0} is clearly equivalent to condition (i).

We claim that~\eqref{eq:inequality-KKT-proof-1} implies 
\begin{equation}
\label{eq:inequality-KKT-proof-2}
\langle q_2(\omega), \y - \by(\omega) \rangle_{\Y^*,\Y} \geq 0 \quad \forall \y \in C_2 \text{ a.s. } 
\end{equation}
Let $\hat{C}_2$ be a countable dense subset of $C_2$. For
$\y \in \hat{C}_2$, we define
\begin{equation*}
  \tilde{x}_2(\omega):= \begin{cases}
 \y, & \text{if } \langle q_2(\omega), \y - \by(\omega) \rangle_{\Y^*,\Y} < 0\\
 \by(\omega), & \text{otherwise }  
  \end{cases}.
\end{equation*}
The function $\tilde{x}_2$ is clearly in $L^\infty(\Omega,\Y)$ and satisfies $\tilde{x}_2(\omega) \in C_2$ a.s.
Since~\eqref{eq:inequality-KKT-proof-1} holds we have
\begin{align*}
0 &\leq\E[\langle q_2(\cdot), \tilde{x}_2(\cdot) - \by(\cdot) \rangle_{\Y^*,\Y}] =\E[ \min(0,\langle q_2(\cdot), \y-\by(\cdot)\rangle_{\Y^*,\Y}) ],
\end{align*}
which gives $\langle q_2(\omega), \y-\by(\omega) \rangle_{\Y^*,\Y} \geq 0$
a.s. Since this is true for all $\y \in \hat{C}_2$ and $\hat{C}_2$ is
countable, there exists a set $\Omega'\subset \Omega$ such that
$\pP(\Omega')=1$ and
\begin{equation*}
  \langle q_2(\omega), \y-\by(\omega) \rangle_{\Y^*,\Y} \geq 0 \quad \forall \y \in \hat{C}_2 \text{ and } \forall \omega \in \Omega'.
\end{equation*}
Passing to the closure of $\hat{C}_2$, we get
\begin{equation*}
  \langle q_2(\omega), \y-\by(\omega) \rangle_{\Y^*,\Y} \geq 0 \quad \forall \y \in C_2 \text{ and } \forall \omega \in \Omega',
\end{equation*}
and hence we have shown~\eqref{eq:inequality-KKT-proof-2}.

Finally,~\eqref{eq:subdifferential-h2} implies
with~\eqref{eq:inequality-KKT-proof-2} that for all
$(x_1,x_2) \in \U \times C_2$,
\begin{equation*}
  h_2(\u,\y,\omega)\geq h_2(\bu, \by(\omega),\omega) + \langle q_1(\omega), \u - \bu \rangle_{\U^*,\U} \quad \text{a.s.}
\end{equation*}
With the definition of $h_2$ given in~\eqref{eq:def-h2}, it follows that
\begin{equation}
\label{eq:inequality-KKT-proof-3}
\begin{aligned}
&J_2(\u,\y) + \langle \bp, e(\u,\y,\omega)\rangle_{W^*,W} + \langle \bl, i(\u,\y,\omega) \rangle_{R^*,R} - \langle q_1(\omega), \u\rangle_{\U^*,\U} \\
&\quad\geq J_2(\bu,\by(\omega)) + \langle \bp,e(\bu,\by(\omega),\omega) \rangle_{W^*,W} \\
& \quad\quad  \quad \quad+ \langle \bl, i(\bu, \by(\omega),\omega)\rangle_{R^*,R} - \langle q_1(\omega), \bu\rangle_{\U^*,\U}
\end{aligned}
\end{equation}
for all $(x_1,x_2) \in \U \times C_2.$
The inequality~\eqref{eq:inequality-KKT-proof-3} is clearly equivalent to
condition (ii) with $\rho(\omega):=q_1(\omega)$. 

\section{Model Problem with Almost Sure State Constraints}
\label{sec:model-problem}
Before we proceed to a concrete example, we will discuss a particular class
of problems that will help us in verifying the measurability requirements
posed in  \cref{assumption:general-problem}. Let $\mathcal{L}(Y,W)$
denote the space of all bounded linear operators from $Y$ to $W$. A random
linear operator $\cA:\Omega \rightarrow \mathcal{L}(Y,W)$ is called strongly
measurable if for all $y\in Y$ the $W$-valued random variable
$\omega \mapsto \cA(\omega)y$ is strongly measurable.
Let $\cA:\Omega \rightarrow \mathcal{L}(Y,W)$,
$\cB:\Omega \rightarrow \mathcal{L}(\U,W)$, and $g:\Omega \rightarrow W$
be (strongly) measurable random operators. We consider the random linear
operator equation
\begin{equation}
\label{eq:random-PDE-abstract}
\mathcal{A}(\omega) y = \mathcal{B}(\omega)\u + g(\omega).
\end{equation} 
The inverse and adjoint operators are to be understood in the
``almost sure'' sense; e.g., for $\cB$, the adjoint operator is the random
operator $\cB^*$ such that for all $(\u, w^*) \in \U \times W^*$, 
\begin{align*}
\pP(\{\omega \in \Omega: \langle w^*, \cB(\omega)\u \rangle_{W^*,W} =  \langle  \cB^{*}(\omega)w^*,\u \rangle_{\U^*,\U} \}) = 1.
\end{align*}

The following theorem will help us verify measurability in the application.
% and can be attributed to Hans }.%, Hans1961}.

\begin{theorem}[Hans~\cite{Hans1957}]
\label{thm:Hans}
Let $\cA:\Omega \rightarrow \mathcal{L}(Y,W).$ Then $\cA(\omega)$ is
invertible a.s if and only if $\textup{ran}(\cA^*(\omega)) = Y^*$ a.s. If
these conditions are satisfied, then $\cA^*(\omega)$ is invertible and
$(\cA^*(\omega))^{-1} = (\cA^{-1}(\omega))^*.$ Moreover, if any of the
operators $\cA(\omega)$, $\cA^{-1}(\omega)$, $\cA^*(\omega)$,
$ (\cA^{-1}(\omega))^*$  is measurable, then all four operators are
measurable.
\end{theorem} 

If $\cA(\omega) \in \mathcal{L}(Y,W)$ is a linear isomorphism for almost
every $\omega$, then $\cA(\omega)$ is invertible and
$\cA^{-1}(\omega) \in \mathcal{L}(W,Y).$ The existence and uniqueness of
the solution to~\eqref{eq:random-PDE-abstract}, given by
\begin{equation*}
  y(\omega) = \cA^{-1}(\omega)(\cB(\omega) x_1 + g(\omega)) \in Y,
\end{equation*}
follows. By \cref{thm:Hans}, $\cA^{-1}(\omega)$ is measurable, hence
$y$ is strongly measurable as a product of strongly measurable functions;
see~\cite[Proposition~1.1.28, Corollary 1.1.28]{Hytoenen2016}.

\subsection{Example}
\label{subsection:Example}
Let $D \subset \R^2$ be a bounded Lipschitz domain. $W^{1,p}(D)$ denotes the
(reflexive and separable) Sobolev space on $D$ consisting of functions in
$L^p(D)$ having first-order distributional derivatives also in $L^p(D)$.
$W_{0}^{1,p}(D)$ is the subset of functions in $W^{1,p}(D)$ that vanish on
the boundary $\partial D$. Additionally, $W^{-1,p}(D)$ denotes the dual
space of $W^{1,p'}_0(D)$, where $\nicefrac{1}{p}+\nicefrac{1}{p'} = 1.$

We set $X_1 = L^2(D)$, $Y = W_0^{1,p}(D)$, for some suitable $p>2$, and let $C_1 \subset X_1$
and $C_2\subset Y$ be nonempty, convex, and closed sets. The inner product on $X_1$ is denoted by $(\cdot,\cdot)_{X_1}$. Given a target
$y_D \in \U$, a constant $\alpha >0$, and a constraint
$\psi \in L^\infty(\Omega,Y)$, the problem is 
\begin{equation}
\label{eq:model-PDE-UQ-state-constraints}
\tag{$\textup{P}'$}
 \begin{aligned}
&\min_{(x_1, y) \in X_1 \times L^\infty(\Omega,Y)}\quad \frac{1}{2} \E \left[ \lVert y - y_D \rVert_{\U}^2 \right]+ \frac{\alpha}{2} \lVert \u \rVert_{\U}^2 \\
&\quad  \text{s.t.} \left\{\begin{aligned}
x_1 &\in C_1, \\
y(\cdot,\omega) &\in C_2 \text{ a.s.},\\
  -\nabla \cdot (a(s,\omega) \nabla y(s,\omega)) &=  \u(s) + g(s,\omega)& \text{ on }& D \times \Omega \text{ a.e.}, \\
  y(s,\omega) &= 0 &\text{ on }& \partial D \times \Omega \text{ a.e.},\\
  y(s,\omega) &\leq \psi(s,\omega)& \text{ on }& D \times \Omega \text{ a.e.},
\end{aligned}\right.
 \end{aligned}
\end{equation}
where ``a.e.'' signifies almost everywhere in $D$ and almost surely in
$\Omega$. We note that the solution to the PDE is a random field
$y:\Omega \times D \rightarrow \R$; we use the shorthand
$y_\omega:=y(\cdot,\omega)$ to denote a single realization. The random
fields $a:D \times \Omega \rightarrow \R$ and
$g:D\times \Omega \rightarrow \R$ are subject to the following assumption.
\begin{assumption}
\label{assumption:regularity-random-field}
The function $g$ satisfies $g \in L^\infty(\Omega,L^2(D))$.
There exist $a_{\min}, a_{\max}$ such that 
$0 < a_{\min} \le a(s,\omega) \le a_{\max} < \infty$ a.e. on $D\times \Omega$.
Additionally, $a \in L^\infty(\Omega,C^t(D))$ for some $t \in (0,1]$.
\end{assumption}
It will be useful to define the (self-adjoint) operators 
\begin{equation*}
  \cA(\omega)y := b_\omega(y,\cdot) \quad \text{for} \quad
  b_\omega(y,\phi) := \int_D a(\cdot,\omega) \nabla y \cdot \nabla \phi \D s
\end{equation*}
and $\cB(\omega) := \textup{id}_{\U}.$ We first address the solvability of
the random PDE in Problem~\eqref{eq:model-PDE-UQ-state-constraints}. 
\begin{lemma}
\label{lemma:existence-uniqueness-model-problem}
Under  \cref{assumption:regularity-random-field}, there exists
$p > 2$ such that for all $\u \in \U$ and almost every $\omega \in \Omega$,
there exists a unique $y_\omega=y(\cdot,\omega) \in Y.$ Furthermore,
$y \in L^\infty(\Omega,Y)$.
\end{lemma}

\begin{proof}
Due to  \cref{assumption:regularity-random-field}
and~\cite{Groeger1989} there exists
some $p > 2$ such that, a.s.,
$\mathcal{A}(\omega) \colon Y = W^{1,p}_0(D)\rightarrow W^{-1,p}(D)$
is an isomorphism and 
\begin{equation*}
\|\mathcal{A}^{-1}(\omega)\|_{\mathcal{L}(W^{1,p}_0(D),W^{-1,p}(D))} \le c
\end{equation*}
for a constant $c$ independent of $\omega$.

Now, since $D \subset \R^2$, $L^2(D) \subset W^{-1,p}(D)$ for all
$p < \infty$ and thus
\begin{equation*}
y_\omega = \mathcal{A}(\omega)^{-1}(\mathcal{B}(\omega)\u + g(\cdot,\omega)) \in Y
\end{equation*}
is well-defined with
$\mathcal{B}\colon L^2(D) \rightarrow L^\infty(\Omega, L^2(D))$ being the
mapping to constant functions in $\Omega$.

Clearly, it holds a.s.
\begin{align*}
  \|y_\omega\|_Y &\le \|\mathcal{A}^{-1}(\omega)\|_{\mathcal{L}(W^{1,p}_0(D),W^{-1,p}(D))} \|\mathcal B(\omega) \u + g(\omega)\|_{W^{-1,p}(D)}\\
  &\le c (\|\u\|_{W^{-1,p}(D)}+\|g(\omega)\|_{W^{-1,p}(D)}) \\
  &\le c (\|\u\|_{L^2(D)} + \|g\|_{L^{\infty}(\Omega;L^2(D))})
\end{align*}
Strong measurability of $y$ follows as argued after \cref{thm:Hans}.
\end{proof}

To obtain necessary and sufficient KKT conditions, we first note that unless
the constraint $\y(s,\omega) \leq \psi(s,\omega)$ is trivially satisfied
almost surely, Problem~\eqref{eq:model-PDE-UQ-state-constraints} does not
satisfy the relatively complete recourse
condition~\eqref{eq:relatively-complete-recourse}. It therefore makes
sense to modify the model to ensure that the second-stage problem is always
feasible. We introduce a slack variable $z \in Y$ and constant $\alpha'>0$;
the second-stage variable is then defined by
$x_2=(y,z) \in X_2:= L^\infty(\Omega,Y) \times L^\infty(\Omega,Y)$.
This modified problem is
\begin{equation}
\label{eq:model-PDE-UQ-state-constraints-slack}
\tag{$\text{P$'_s$}$}
 \begin{aligned}
&  \min_{(\u,x_2) \in X_1 \times X_2}\quad  \frac{1}{2} \E \left[ \lVert y - y_D \rVert_{\U}^2  +\alpha' \lVert z \rVert_{\U}^2 \right]+ \frac{\alpha}{2} \lVert \u \rVert_{\U}^2 \\
   &\quad  \text{s.t.} \left\{\begin{aligned}
   x_1 &\in C_1, \\
   y(\cdot,\omega) &\in C_2 \text{ a.s.},\\
   z(\cdot,\omega) &\in C_2 \text{ a.s.},\\
  -\nabla \cdot (a(s,\omega) \nabla y(s,\omega)) &=  \u(s) + g(s,\omega) &\text{ on }& D \times \Omega \text{ a.e.}, \\
  y(s,\omega) &= 0& \text{ on }& \partial D \times \Omega \text{ a.e.},\\
  y(s,\omega) &\leq \psi(s,\omega) + z(s,\omega)&\text{ on }& D \times \Omega \text{ a.e.}
   \end{aligned}\right.
 \end{aligned}
\end{equation}
It is clear that Problem~\eqref{eq:model-PDE-UQ-state-constraints-slack}
now satisfies the condition~\eqref{eq:relatively-complete-recourse} of
relatively complete recourse if $C_2$ and $\psi$ are chosen appropriately.
For example, by
\cref{lemma:existence-uniqueness-model-problem}, one immediately
obtains a unique solution $y$ to the PDE constraint where
$\|y\|_{L^\infty(\Omega, Y)} \le c$ whenever $C_1$ is bounded in $L^2(D)$.
Then, if $C_2$ is a sufficiently large ball $z:=\psi - y$ is again in
$C_2$ and thus
the pair $(y,z)$ is feasible. 

In this model, we have
\begin{align*}
J_1(\u) &=\frac{\alpha}{2} \lVert \u \rVert_{\U}^2,\\
J_2(\u, \y) &= \frac{1}{2}  \lVert y - y_D \rVert_{\U}^2 + \frac{\alpha'}{2} \lVert z \rVert_{\U}^2,\\
e(\u,\y,\omega) &= \cA(\omega) y - \cB(\omega) \u - g(\cdot,\omega) \in Y^*,\\
i(\u,\y,\omega) & = y - \psi(\cdot,\omega) - z \in Y,\\
K & = \{ y \in Y: y(s) \geq 0  \quad \text{on } D \text{ a.e.}\}.
\end{align*}
It is clear that  \cref{assumption:general-problem} is satisfied
here. Indeed, $J(\u,\y) = J_1(\u) + J_2(\u,\y)$ is convex, everywhere
defined, and continuous in $X_1 \times X_2$. The function $e(\u,\y,\omega)$
is linear and continuous in $(\u,\y)$; measurability follows from the
assumed measurability of the underlying operators. Additionally,
$i(\u,\y,\omega)$ is linear and continuous in $\y$ as well as measurable
since $\psi \in L^\infty(\Omega,Y)$. 

Now, we can formulate KKT conditions for
Problem~\eqref{eq:model-PDE-UQ-state-constraints-slack}.

\begin{lemma}
\label{lemma:KKT-general}
Suppose  \cref{assumption:regularity-random-field} is satisfied
and $C_1, C_2$ are bounded. Then $(\bx,\blam)$ is a saddle point of the
Lagrangian~\eqref{eq:Lagrangian} for
Problem~\eqref{eq:model-PDE-UQ-state-constraints-slack} if and only if
there exist $\rho \in L^1(\Omega,\U^*)$, $\bp \in L^1(\Omega,Y)$, and
$\bl \in L^1(\Omega, Y^*)$ such that for all $x_1 \in C_1$ and all $(y,z) \in C_2\times C_2,$
\begin{subequations}
\begin{align}
( \alpha \bar{x}_1 + \E[\rho], x_1 - \bar{x}_1 )_{\U} &\geq 0, \label{eq:FOC-PDE-problem1}\\% \quad \forall x_1 \in C_1,
\cB^*(\omega) \bar{\lambda}_{e,\omega} +\rho_\omega & = 0,\label{eq:FOC-PDE-problem2}\\
 (\bar{y}_\omega-y_D, y-\bar{y}_\omega)_{X_1}  +\langle \cA^*(\omega)\bar{\lambda}_{e,\omega} +\bar{\lambda}_{i,\omega}, y-\bar{y}_\omega\rangle_{Y^*,Y} &\geq 0,\label{eq:FOC-PDE-problem3}\\
  ( \alpha' \bar{z}_\omega, z-\bar{z}_{\omega})_{X_1}
  -\langle \bar{\lambda}_{i,\omega}, z-\bar{z}_\omega\rangle_{Y^*,Y} &\geq 0,\label{eq:FOC-PDE-problem3'}\\
  \cA(\omega) \bar{y}_\omega - \cB(\omega) \bu - g_\omega &= 0,\label{eq:FOC-PDE-problem4}\\
\bar{\lambda}_{i,\omega} \in K^{\oplus}, \quad \bar{y}_\omega-\bar{z}_\omega \leq_K \psi_\omega, \quad
 \langle \bar{\lambda}_{i,\omega}, \bar{y}_\omega-\bar{z}_\omega-\psi_\omega \rangle_{Y^*,Y} &= 0,\label{eq:FOC-PDE-problem5}
\end{align}
\end{subequations}
where~\eqref{eq:FOC-PDE-problem2}--\eqref{eq:FOC-PDE-problem5} hold for
almost all $\omega \in \Omega.$ These conditions are necessary and
sufficient for optimality.
\end{lemma}

\begin{proof}
We apply the optimality conditions (i)--(iii) from
\cref{thm:KKT-basic-conditions}. Let 
$f_1(\u):=J_1(\u) + \langle \E[\rho],\u \rangle_{\U^*,\U}$. We recall that
the optimum $\u$ over $C_1$ is attained if and only if
$\langle f_1'(\bu), \u - \bu \rangle_{\U^*,\U} \geq 0$ for all $\u \in C_1$.
Hence condition (i) is equivalent to~\eqref{eq:FOC-PDE-problem1}.
Now, we define
\begin{align*}
f_2(\u,\y,\omega) &:=J_2(\u,\y) + \langle \bar{\lambda}_{e,\omega}, e(\u,\y,\omega) \rangle_{Y,Y^*}\\
&\quad + \langle \bar{\lambda}_{i,\omega}, i(\u,\y,\omega) \rangle_{Y^*,Y} - \langle \rho_\omega,\u\rangle_{\U^*,\U}.
\end{align*}
Now, (ii) is equivalent to stationarity of $f_2$ yielding~\eqref{eq:FOC-PDE-problem2}--\eqref{eq:FOC-PDE-problem3'}.
To see this, we compute $$D_{\u} f_2(\u,\y(\omega),\omega)[h] = \langle- \cB^*(\omega) \bar{\lambda}_{e,\omega} - \rho_\omega, h \rangle_{\U^*,\U},$$
so $D_{\u} f_2(\bu,\by(\omega),\omega) = 0$ a.s.~if and only
if~\eqref{eq:FOC-PDE-problem2} holds. 
Recalling that $x_2 = (y,z)$, we compute
\begin{align*}
D_{y} f_2(\u,\y(\omega),\omega)[k_1] &= ( y_\omega - y_D, k_1)_{X_1} + \langle\cA^*(\omega) \bar{\lambda}_{e,\omega} +\bar{\lambda}_{i,\omega}, k_1 \rangle_{Y^*,Y},\\
D_{z} f_2(\u,\y(\omega),\omega)[k_2] &= (\alpha' z_\omega, k_2)_{X_1} 
 -\langle \bar{\lambda}_{i,\omega}, k_2 \rangle_{Y^*,Y},
\end{align*} 
which at the optimum $\by=(\bar{y},\bar{z})$ over $C_2\times C_2$ is equivalent 
to~\eqref{eq:FOC-PDE-problem3}--\eqref{eq:FOC-PDE-problem3'}.
Condition (iii) is clearly equivalent to~\eqref{eq:FOC-PDE-problem4}
and \eqref{eq:FOC-PDE-problem5}. 

For the final statement, it suffices to verify that
Problem~\eqref{eq:model-PDE-UQ-state-constraints-slack} is strictly
feasible. 
Since $p>2$ and $D \subset \R^2$ is bounded, $W^{1,p}(D)$ is compactly
embedded in $C(\bar{D})$. Note that $y_\omega, z_\omega \in W^{1,p}(D)$
satisfying
\begin{equation*}
  i(\u,\y(\omega),\omega)=y_\omega-\psi(\cdot,\omega) - z_\omega <_K 0
\end{equation*}
means $\eta_\omega(s):=y_\omega(s) -\psi(s,\omega)-z_\omega(s)  < 0 \quad \text{a.e. on } \bar{D}.$
Now, the continuous function $\eta_\omega$ must take its maximum on the
compact set $\bar{D}$, so  there exists a $\varepsilon=\varepsilon(\omega)$
such that 
$\eta_\omega=i(\u,\y(\omega),\omega) <-\varepsilon$ a.e.~on $\bar{D}.$
If $v_\omega \in W^{1,p}(D)$ is chosen such that
$\|v_\omega\|_{\infty} \leq \delta(\omega)$, then
\begin{align*}
i(\u,\y(\omega)+v_\omega,\omega) &= i(\u,\y(\omega),\omega) +v_\omega \leq  -\varepsilon + \|v_\omega\|_\infty \leq -\varepsilon + \delta(\omega)
\end{align*}
and therefore
$i(\u,\y(\omega),\omega)<v_\omega$ if $\delta(\omega) < \epsilon.$
By \cref{thm:minP-supD} and \cref{thm:infP-maxD}, these
conditions are necessary and sufficient.
\end{proof}

Finally, let us note that taking $\alpha' \rightarrow \infty$, the
primal variables of Problem~\eqref{eq:model-PDE-UQ-state-constraints-slack}
converge to those of Problem~\eqref{eq:model-PDE-UQ-state-constraints}
assuming that the latter has a solution.
\begin{theorem}
Assume that $0 \in C_2$, Problem~\eqref{eq:model-PDE-UQ-state-constraints} has at
least one  optimal solution $(x_1,y) \in \U \times L^\infty(\Omega,Y)$ and let
$(x_1^{\alpha'},y^{\alpha'},z^{\alpha'}) \in \U \times L^\infty(\Omega,Y) \times
L^\infty(\Omega,Y)$ be solutions to Problem~\eqref{eq:model-PDE-UQ-state-constraints-slack}.
Then for any sequence $\{\alpha_n'\}$ such that $\alpha_n' \rightarrow \infty$, the sequence
$\{ (x_1^{\alpha_n'},y^{\alpha_n'}) \}$ has a (weak, strong) accumulation point
$(x_1^{\infty},y^{\infty})$,
i.e., $x_1^{\alpha_n'} \rightharpoonup x_1^\infty$ in $\U$ and
$y^{\alpha_n'}\rightarrow y^\infty $ in $L^\infty(\Omega,Y)$,
and each such accumulation point solves Problem~\eqref{eq:model-PDE-UQ-state-constraints}.
\end{theorem}
\begin{proof}
  Denote by $    J^{\alpha'}(x_1,y,z) = \frac{1}{2} \E \left[ \lVert y - y_D \rVert_{\U}^2  +\alpha' \lVert z \rVert_{\U}^2 \right]+ \frac{\alpha}{2} \lVert \u \rVert_{\U}^2$
  the objective of
  Problem~\eqref{eq:model-PDE-UQ-state-constraints-slack}.
  The objective $ J^\infty(x_1,y) = \frac{1}{2} \E \left[ \lVert y - y_D \rVert_{\U}^2 \right]+ \frac{\alpha}{2} \lVert \u \rVert_{\U}^2$ corresponds to Problem~\eqref{eq:model-PDE-UQ-state-constraints}.
  By assumption, $(x_1,y,0)$ is a feasible point for all $\alpha'$ and
  thus
  \[
    J^{\alpha'}(x_1^{\alpha'},y^{\alpha'},z^{\alpha'}) \le
    J^{\alpha'}(x_1,y,0) = J^\infty(x_1,y).
  \]
  Consequently $x_1^{\alpha'}$ is bounded in $\U$ and $\E[\lVert
  z^{\alpha_n'} \rVert_{\U}^2] \rightarrow 0$. By convexity and
  closedness of $C_1$ we have a weakly convergent subsequence 
  (again denoted by $x_1^{\alpha'_n}$) such that
  $x_1^{\alpha'_{n}} \rightharpoonup x_1^{\infty} \in C_1$.
  By compactness of the embedding $\U \subset Y^*$ this implies strong convergence, of the same subsequence,
  $y^{\alpha'_n} \rightarrow y^\infty\in C_2$ in $L^\infty(\Omega,Y)$
  and by linearity of the PDE $x_1^\infty$ and $y^\infty$ solve the
  PDE. Convergence of $y^{\alpha'_n} \rightarrow y^\infty$ and $\E[\lVert
  z^{\alpha'_n} \rVert_{\U}^2] \rightarrow 0$ show that the inequality $ y^\infty(s,\omega) \le \psi(s, \omega)$
  holds true a.e.~on $D \times \Omega$. Consequently the limit is
  feasible for Problem~\eqref{eq:model-PDE-UQ-state-constraints}.
  Weak lower semicontinuity of $\lVert\cdot \rVert_{\U}$ shows 
  \[
    J(x_1^\infty,y^\infty) \le J(x_1^{\alpha'},y^{\alpha'}) \le
    J_{\alpha'}(x_1^{\alpha'},y^{\alpha'},z^{\alpha'}) \le J(x_1,y)
  \]
  and thus $(x_1^\infty,y^\infty)$ is a solution of
  Problem~\eqref{eq:model-PDE-UQ-state-constraints}.

  Clearly the argument holds for any such convergent subsequence.  
\end{proof}

\subsection{Outlook}
In addition to the Problem~\eqref{eq:model-PDE-UQ-state-constraints-slack}, there are a number of other potential applications to our theory.
For instance, in the optimal control of ordinary differential equations
(ODEs) with uncertainties (cf.~\cite{Phelps2016}), the addition of a
constraint on the state would also require essentially bounded states
in order to satisfy constraint qualifications. To use a sample average
approximation (SAA) in their work, optimality conditions were needed and
our theory could also be of use in their linear example (the design of a
control to stabilize a harmonic oscillator). For applications to shape
optimization under uncertainty
(cf.~\cite{Atwal2012, Conti2008, Geiersbach2021a}), it is certainly
desirable in, e.g., a linear elasticity model to require pointwise bounds
on the solution to the corresponding PDE, which represents the displacement
field of a shape. Here, the control-to-state mapping is nonlinear and therefore our theory 
is not immediately applicable; further research would be desirable. In the development of
algorithms, we note that for (deterministic, infinite-dimensional) state
constraints, penalty methods are frequently employed due to unruly
singular terms arising in KKT conditions. Therefore penalizing almost
sure state constraints the way we propose in the previous section is
quite natural and could easily be modified for the above-mentioned
applications in the optimal control of ODEs with uncertainties and shape
optimization. Additionally, by \cref{rem:discrete-probability-space},
the relatively complete recourse condition is also unproblematic as soon
as one uses an SAA approximation or the underlying model has only finitely
many scenarios as in \cite{Atwal2012}.

\section{Conclusion}
\label{sec:Conclusion}
In this paper, we focused on obtaining necessary and sufficient first-order
optimality conditions for a class of stochastic convex optimization
problems. The first stage variable $x_1$ was assumed to belong to a
reflexive and separable Banach space, and the second-stage variable $x_2$
was assumed to be an essentially bounded random variable having an image in
a reflexive and separable Banach space. While the study of such problems in
finite dimensions is classical, going back to a series of papers from the
1970s by Rockafellar and Wets, its treatment in Bochner spaces, although
cursorily handled in~\cite{Rockafellar1971, Rockafellar1974}, was not
complete enough to handle a class of problems of increasing interest,
namely PDE-constrained optimization under uncertainty. In such problems,
it is desirable to find a control $x_1$ such that a partial differential
equation depending on the control is satisfied. The additional pointwise
constraints on the solution to the PDE presented surprising difficulties.
In order to obtain necessary and sufficient conditions for optimality, we
built on the decomposition result provided by
Ioffe and Levin~\cite{Ioffe1972}, in which
the Bochner space $L^\infty(\Omega,X)$ is decomposed into its absolutely
continuous part and a singular part. We find that the singular part
vanishes in the optimality conditions if strict feasibility and relatively
complete recourse conditions are satisfied. This provides necessary and
sufficient conditions for optimality with integrable Lagrange multipliers.
While the example model problem we chose to illustrate the theory involved
smooth functions, we remark that the optimality conditions do not require
smoothness of the objective functions. Therefore we believe our theory to
be applicable to more general risk-averse problems.

\appendix
\section{Appendix}
\paragraph{Expansion of generalized Lagrangian~\eqref{eq:Lagrangian}} 
If $x \not\in X_0$, then $x \not\in \textup{dom}\, \varphi(\cdot,u)$ and
therefore $L(x,\lambda) = \infty$ by definition
of~\eqref{eq:generalized-Lagrangian-definition}. Now we observe the case $x \in X_0$. The constraint
$i(\u, \y(\omega), \omega) \leq_K u_i(\omega)$ is equivalent to
$u_i(\omega) - i(\u, \y(\omega), \omega) \in K.$ 
Since $x \in X_0$, $\varphi$ can be redefined equivalently by
\begin{equation*}
  \varphi(x,u) := j(x) + \E[\delta_{\{ u_e(\cdot)\}}(e,\u,\y(\cdot),\cdot)] + \E[\delta_K(u_i(\cdot) - i(\u,\y(\cdot),\cdot))].
\end{equation*}
(The equivalence is clear after one notices that the indicator function is
non-negative.) Expanding~\eqref{eq:generalized-Lagrangian-definition},
we get
\begin{align*}
L(x,\lambda) =  j(x) + &\inf_{u \in U} \big\lbrace \E[\delta_{\{ u_e(\cdot)\}}(e(\u,\y(\cdot),\cdot))]\\
&  \quad \quad +\E[\delta_{K}(u_i(\cdot)-i(\u,\y(\cdot),\cdot))] + \langle u,\lambda\rangle_{U,\Z}\big\rbrace.
\end{align*}
Recalling the definition of the pairing~\eqref{eq:dual-pairing}, we first
see that
\begin{equation}
\label{eq:appendix-proof1}
\begin{aligned}
  &\inf_{u_e \in L^\infty(\Omega,W)} \int_{\Omega} \delta_{\{ u_e(\omega)\}}(e(\u,\y(\omega),\omega)) + \langle u_e(\omega),\lambda_e(\omega) \rangle_{W^*,W} \D \pP(\omega) \\
& \qquad= \int_\Omega \langle e(\u,\y(\omega),\omega),\lambda_e(\omega) \rangle_{W^*,W} \D \pP(\omega) \\
& \qquad \quad +\inf_{z \in L^\infty(\Omega,W)} \int_\Omega \delta_{\{0\}}(z(\omega)) - \langle z(\omega),\lambda_e(\omega) \rangle_{W^*,W}\D \pP(\omega)\\
  & \qquad=  \int_\Omega \langle e(\u,\y(\omega),\omega),\lambda_e(\omega) \rangle_{W^*,W} \D \pP(\omega) - \int_{\Omega}\delta^*_{\{0\}} (\lambda_e(\omega)) \D \pP(\omega)\\
& \qquad= \int_{\Omega} \langle e(\u,\y(\omega),\omega), \lambda_e(\omega)\rangle_{W^*,W} \D \pP(\omega),
\end{aligned}
\end{equation}
where in the last step, we used that the conjugate of the indicator
function is equal to the support function. Similarly, 
\begin{equation}
\label{eq:appendix-proof2}
\begin{aligned}
  & \inf_{u_i \in L^\infty(\Omega,R)} \int_{\Omega} \delta_{K}\big(u_i(\omega)-i(\u,\y(\omega),\omega)\big) +\langle u_i(\omega),\lambda_i(\omega)\rangle_{R,R^*} \D \pP(\omega)\\
& \quad = \int_{\Omega} \langle i(\u,\y(\omega),\omega), \lambda_i(\omega) \rangle_{R,R^*} \D \pP(\omega) \\
& \quad \quad - \sup_{z \in L^\infty(\Omega,R)} \int_{\Omega} \delta_K(-z(\omega)) - \langle z(\omega),\lambda_i(\omega) \rangle_{R,R^*} \D \pP(\omega)\\
 & \quad = \int_{\Omega} \langle i(\u,\y(\omega),\omega), \lambda_i(\omega) \rangle_{R,R^*} - \sup_{z' \in -K} \langle z',\lami(\omega)\rangle_{R,R^*} \D \pP(\omega) \\
 & \quad = \int_{\Omega} \langle i(\u,\y(\omega),\omega), \lambda_i(\omega) \rangle_{R,R^*} - \delta_{K^{\oplus}}(\lami(\omega))  \D \pP(\omega).
\end{aligned}
\end{equation}
If $\lambda_i(\omega) \not\in K^{\oplus}$, then the integral is equal to $-\infty$.
Otherwise, if $\lambda \in \Lambda_0$ (and $x \in X_0$), we get after
combining~\eqref{eq:appendix-proof1} and~\eqref{eq:appendix-proof2} the
expression
\begin{equation*}
  L(x,\lambda) = j(x) +\E[\langle e(\u,\y(\cdot),\cdot), \lambda_e(\cdot)\rangle_{W^*,W}] + \E[ \langle i(\u,\y(\cdot),\cdot), \lambda_i(\cdot) \rangle_{R,R^*}].
\end{equation*}

%\section*{Acknowledgments}
\bibliographystyle{abbrvnat}
\bibliography{references}
\end{document}